\documentclass{amsart}
\usepackage{latexsym}
\usepackage{graphicx,subfigure}
\usepackage{array}
\usepackage{indentfirst}
\usepackage{fancyhdr}
\usepackage{epsf}
\usepackage{amsmath,bm}
\usepackage{amsthm}
\usepackage{booktabs}
\usepackage{amsfonts}
\usepackage{epsf}
\usepackage{multicol}
\usepackage[square, comma, sort&compress, numbers]{natbib}
\usepackage{fancyhdr}
\usepackage{pxfonts}
\usepackage{caption}
\usepackage{epstopdf}
\usepackage{algorithm}
\usepackage{algpseudocode}
\usepackage{exscale}
\usepackage{relsize}
\usepackage{setspace}
\usepackage{geometry}
\usepackage{subfigure}
\usepackage{tikz}
\usetikzlibrary{arrows,shapes,chains,positioning}
\usetikzlibrary{decorations.markings}
\tikzstyle arrowstyle=[scale=1]
\tikzstyle directed=[postaction={decorate,decoration={markings,mark=at position .65 with {\arrow[arrowstyle]{stealth}}}}]
\tikzstyle reverse directed=[postaction={decorate,decoration={markings,mark=at position .65 with {\arrowreversed[arrowstyle]{stealth};}}}]

\newtheorem{Def}{Definition}
\newtheorem{Th}{Theorem}[section]
\newtheorem{Lm}{Lemma}[section]

\newtheorem{remark}{Remark}[section]

\newtheorem{theorem}{Theorem}[section]

\theoremstyle{definition}
\usepackage[titletoc]{appendix}
\numberwithin{equation}{section}
\newcommand{\bc}{\begin{center}}
\newcommand{\ec}{\end{center}}
\newcommand{\be}{\begin{eqnarray}}
\newcommand{\ee}{\end{eqnarray}}

\newcommand{\ben}{\begin{eqnarray*}}
\newcommand{\een}{\end{eqnarray*}}

\newcommand{\Om}{{\rm\Omega}}

\newcommand{\dx}{\,dx}

\newcommand{\ds}{\,ds}

\newcommand{\Rmnum}[1]{\expandafter\@slowromancap\romannumeral #1@}
\newcommand{\Rmath}{\mathbb{R}}
\newcommand{\Smath}{\mathbb{S}}
\newcommand{\Pma}{\mathcal{P}}
\newcommand{\Tma}{\mathcal{T}_h}

\newcommand{\sigRTS}{\sigma_{\rm RT}}
\newcommand{\uRTS}{u_{\rm RT}}
\newcommand{\PiRT}{\Pi_{\rm RT}}
\newcommand{\uS}{u}
\newcommand{\sigS}{\sigma}

\newcommand{\nBi}{\bm{n}}
\newcommand{\sSpaMBi}{S}
\newcommand{\uSpaMBi}{D}

\newcommand{\ubi}{u_\Pma}
\newcommand{\ubiM}{u_\Pma}
\newcommand{\sbiM}{\sigma_\Pma}

\newcommand{\SHHJ}{{\rm HHJ}(\Tma)}
\newcommand{\UHHJ}{U_{\rm HHJ}(\Tma)}
\newcommand{\sHHJ}{\sigma_{\rm HHJ}}
\newcommand{\uHHJ}{u_{\rm HHJ}}
\newcommand{\uMorley}{u_{\rm M}}

\newcommand{\conPone}{S_h}
\newcommand{\PiHHJ}{\Pi_{\rm HHJ}}
\newcommand{\bfv}{\textbf{f}}
\newcommand{\btv}{\textbf{t}}

\newcommand{\cE}{\mathcal{E}}

\newcommand{\cT}{\mathcal{T}}
\newcommand{\R}{\mathbb{R}}

\newcommand{\cP}{\ensuremath{\mathcal{P}} }

\newcommand{\yemeifont}{\fontsize{9pt}{\baselineskip}\selectfont}
\begin{document}
\title{
Optimal Superconvergence Analysis for the Crouzeix-Raviart and the Morley elements
}

\author {Jun Hu}
\address{LMAM and School of Mathematical Sciences, Peking University,
  Beijing 100871, P. R. China.  hujun@math.pku.edu.cn}

\author{Limin Ma}
\address{Department of Mathematics, Pennsylvania State University, University Park, PA,
16802, USA. maliminpku@gmail.com}

\author{Rui Ma}
\address{Institut f\"ur Mathematik, Humboldt-Universit\"at zu Berlin, 10099 Berlin. maruipku@gmail.com}
\thanks{The authors were supported by  NSFC
projects 11625101, 91430213 and 11421101}

\maketitle

\begin{abstract}
In this paper, an improved superconvergence analysis is presented for both the Crouzeix-Raviart element and the Morley element.  The main idea of the analysis  is to employ a discrete Helmholtz decomposition of the difference between the canonical interpolation and the finite element solution for the first order mixed Raviart--Thomas element and the mixed Hellan--Herrmann--Johnson element, respectively.
This in particular allows for proving a full one order superconvergence result for these two mixed finite elements.
Finally, a full one order superconvergence result of both the Crouzeix-Raviart element and the Morley element follows from
their special relations with the first order mixed Raviart--Thomas element and the mixed Hellan--Herrmann--Johnson element respectively. Those superconvergence results are also extended to mildly-structured meshes.

  \vskip 15pt


 \vskip 15pt


\end{abstract}

\section{Introduction}
The superconvergence of  both lower order conforming finite elements and mixed finite elements is well analyzed for second order elliptic problems, see for instance, \cite{chenhigh,Chen2013Superconvergence,Brandts1994Superconvergence,Jan2000Superconvergence,Douglas1989Superconvergence} and the references therein. However, for nonconforming elements, the reduced continuity of both trial and test functions makes the corresponding superconvergence analysis very difficult. So far, most of superconvergence analysis results for nonconforming elements are focused on  rectangular or nearly parallelogram triangulations, see \cite{Zhong2005CONSTRAINED,Lin2005On,Ming2006Superconvergence}. There are a few superconvergence results for nonconforming elements on triangular meshes \cite{Hu2016Superconvergence,li2017global,mao2009high}. In \cite{Hu2016Superconvergence}, a  half order superconvergence was analyzed for the Crouzeix-Raviart (CR for short hereinafter) element and the Morley element. The main idea therein is to employ a special relation between the CR element and the Raviart--Thomas (RT for short hereinafter) element, and the equivalence between the Morley element and the Hellan--Herrmann--Johnson (HHJ for short hereinafter) element to explore some conformity of discrete stresses by these two nonconforming elements. However, a full one order superconvergence was observed in the numerical tests \cite{Hu2016Superconvergence}. Such a gap is caused by a half order superconvergence result for both the RT element \cite{Brandts1994Superconvergence} and the HHJ element \cite{Hu2016Superconvergence}, which is a half order lower than the optimal superconvergence indicated by numerical tests. It is stressed that the superconvergence  analysis of the first order RT element in \cite{Brandts1994Superconvergence}
was heavily dependent on a result of Sobolev spaces and directly used it to estimate one key sum of boundary terms.
Since a counter example \cite{J1972Non} shows that this result of Sobolev spaces can not be improved,
 it is indeed difficult to refine the former superconvergence result within the analysis of \cite{Brandts1994Superconvergence}.
  In \cite{li2017global}, the superconvergence analysis of \cite{bank2003asymptotically} for the conforming linear element was extended to the mixed finite  element, which proved a full one order superconvergence result
  for the first order RT element method of the Poisson problem under the condition  that the solution of the problem is in $H^{4+\epsilon}(\Om,\mathbb{R})$ for any $\epsilon >0$.

In this paper, a new analysis for the aforementioned boundary terms is presented, which leads to  a full one order superconvergence result for both the RT element and the HHJ element on uniform meshes and  improves the corresponding half order superconvergence result in \cite{Brandts1994Superconvergence} and \cite{Hu2016Superconvergence}, respectively. The main ingredient of such a superconvergence analysis is to employ a discrete Helmholtz decomposition of the difference between the canonical interpolation and the finite element solution of the corresponding mixed element. In particular, it allows for some vital cancellation between the boundary terms sharing a common vertex in one key sum. Then,  the final improved superconvergence result follows from  the analysis in \cite{Hu2016Superconvergence} for both the CR element of the Poisson problem and  the Morley element of the plate bending model problem.   Without using variational error expansions  in \cite{li2017global,bank2003asymptotically}, the superconvergence results can be easily generalized to mildly structured piecewise $(\alpha,\sigma)$-meshes in \cite{bank2003asymptotically,HuangXu2008,li2017global}. Moreover, the   mesh-size condition is employed in this paper such that  \begin{align*}
\big|\ln h_K\big|\approx\big|\ln h|\text{ for all }K\in\mathcal{T}_h
\end{align*}
with $h_K={\rm diam}K$ and $h=\max_{K\in\mathcal{T}_h}h_K$. This assumption is weaker than the    quasi-uniformity assumption  in \cite{li2017global}.

The remaining paper is organized as follows. Some notations are presented in Section 2. In Section 3, a full one order superconvergence result for both the RT element and the CR element is proved. In Section 4, a full one order superconvergence result for both the HHJ element and the Morley element is proved. This section also investigates the superconvergence result on  mildly structured $(\alpha,\sigma)$-meshes. Some numerical tests are presented  to verify the theoretical results in Section 5.

\section{Notation}
Given a nonnegative integer $k$ and a bounded domain $\Om\subset \mathbb{R}^2$ with Lipschitz boundary $\partial \Om$, let $W^{k,\infty}(\Om,\mathbb{R})$, $H^k(\Om,\mathbb{R})$, $| \cdot|_{k,\infty, \Om}$, $\parallel \cdot \parallel_{k,\Om}$ and $|\cdot |_{k,\Om}$ denote the usual Sobolev spaces, norm, and semi-norm, respectively. Let $\nBi$ denote the outnormal of $\partial\Omega$,   $H_0^1(\Om,\mathbb{R}) := \{u\in H^1(\Om,\mathbb{R}): u|_{\partial \Om}=0\}$ and $H^2_0(\Omega,\mathbb{R}): =\{u\in H^2(\Om,\mathbb{R}): u|_{\partial \Om}=\partial_{\nBi} u=0\}$. Denote the standard $L^2(\Om,\mathbb{R})$ inner product by $(\cdot, \cdot)$.

Suppose that $\Om\subset \mathbb{R}^2$ is a bounded polygonal domain covered exactly by a shape-regular partition $\cT_h$ into triangles. Let $|K|$ denote the area of element $K$ and $|e|$ the length of edge $e$. Let $h_K$ denote the diameter of element $K\in \cT_h$ and $h:=\max_{K\in\cT_h}h_K$. Denote the set of all interior edges and boundary edges of $\cT_h$ by $\cE_h^i$ and $\cE_h^b$, respectively, and $\cE_h=\cE_h^i\cup \cE_h^b$. For any interior edge $e=K_e^1\cap K_e^2$, denote the element with larger global label by $K_e^1$, the one with smaller global label by $K_e^2$. Denote the corresponding unit normal vector which points from $ K_e^1 $ to $K_e^2$ by $\bold{n}_e$. Let $[\cdot]$ be the jump of piecewise functions over edge $e$, namely
$$
[v]|_e := v|_{K_e^1}-v|_{K_e^2}$$
for any piecewise function $v$. For $K\subset\R^2,\ r\in \mathbb{Z}^+$, let $P_r(K, \mathbb{R})$ be the space of all polynomials of degree not greater than $r$ on $K$. Denote the piecewise gradient operator and the piecewise hessian operator by $\nabla_h$ and $\nabla_h^2$, respectively. For any piecewise function $v_h$, denote $\| v_h\|_{0,\infty, h}=\max_{K\in \cT_h} \| v_h\|_{0,\infty, K}$.

Throughout this paper,  except   in  Subsect.~4.3 for mildly  structured $(\alpha,\sigma)$-meshes, the superconvergence results require triangulations to be uniform, which means that any two adjacent triangles of $\cT_h$ form a parallelogram. 

Recall some notation in \cite{Brandts1994Superconvergence}.
For any triangle $K\in\cT_h$, from the three outer unit normal vectors, denote the two which are closest to orthogonal by $\bold{f}_1$ and $\bold{f}_2$. This procedure is in general not unique, however, only the directions of vectors are focused, thus there will be no restriction.

For each $i=1,\ 2$, denote a parallelogram, which consists of two triangles sharing an edge with normal $\bold{f}_i$, by $N_{\bold{f}_i}$. We partition the domain $\Om$ into those parallelograms $N_{\bold{f}_i} $ and some resulted boundary triangles $K_{\bold{f}_i} $. In an element $K$, denote the edge with unit normal vector $\bold{f}_i$ by $e_{\bold{f}_i}$, the length of $e_{\bold{f}_i}$ by $h_{\bold{f}_i}$, and the unit tangent vector of $e_{\bold{f}_i}$ with counterclockwise by $\bold{t}_{\bold{f}_i}$. We denote the two endpoints of the edge $e_{\bold{f}_i} $ by $\bold{p}_{\bold{f}_i}^1$ and $\bold{p}_{\bold{f}_i}^2$, and $\bold{p}_{\bold{f}_i}^1\bold{p}_{\bold{f}_i}^2 =h_{\bold{f}_i}\bold{t}_{\bold{f}_i}$.
Define
$$
\cP_b:=\big \{\bold{p} \in \partial \Om: \bold{p} \text{ is a vertex of } K_{\bold{f}_i}\big \}.
$$
Decompose the set $\cP_b$ into two parts $ \cP_b=\cP_b^1\cup \cP_b^2$, where $\cP_b^1$ is the set of corners of the domain, and $\cP_b^2$ refers to the remaining vertices. For any vertex $\bold{p}\in \cP_b^1$, denote the unique boundary triangle $K_{\bold{f}_i} $ by $K_\bold{p}$, and for any vertex $\bold{p}\in \cP_b^2$, denote the two boundary triangles $K_{\bold{f}_i} $ sharing $\bold{p}$ by $K^l_\bold{p}$ and $K^r_\bold{p}$, and
$$
K^r_\bold{p}=\{\bold{x}+h_{\bold{f}_i}\bold{t}_{\bold{f}_i}:\bold{x}\in K^l_\bold{p}\}.
$$
For any $\bold{p}\in \cP_b^2$, let $\omega_{\bold{p}}$ denote the trapezoid which consists of three elements and $\bold{p}$ is a midpoint of its edge, see Figure \ref{fig:triangulation}. Let $|\cP_b^1|=\kappa$ denote the number of the vertices in $\cP_b^1$. It is known that $\kappa $ is a fixed number independent of $h$. Figure \ref{fig:triangulation} shows an example of the definitions and notations concerning a triangulation.

\begin{figure}[!ht]
\begin{center}
\begin{tikzpicture}[xscale=6,yscale=6]
\draw[-] (0,0) -- (1.75,0);
\draw[-] (-0.125,0.9) -- (1.625,0.9);
\draw[-] (0.375,0.9) -- (0,0);
\draw[-] (0.625,0.9) -- (0.25,0);
\draw[-] (0.875,0.9) -- (0.5,0);
\draw[-] (1.125,0.9) -- (0.75,0);
\draw[-] (1.375,0.9) -- (1,0);
\draw[-] (1.625,0.9) -- (1.25,0);
\draw[-] (1.75,0.6) -- (1.5,0);
\draw[-] (0.125,0.9) -- (0.5,0);
\draw[-] (0.375,0.9) -- (0.75,0);
\draw[-] (0.625,0.9) -- (1,0);
\draw[-] (0.875,0.9) -- (1.25,0);
\draw[-] (1.125,0.9) -- (1.5,0);
\draw[-] (1.375,0.9) -- (1.75,0);
\draw[-] (1.625,0.9) -- (1.875,0.3);
\draw[-] (-0.125,0.9) -- (0.25,0);
\draw[-] (0.125,0.9) -- (0,0.6);
\draw[-] (0,0.6) -- (1.75,0.6);
\draw[-] (0.125,0.3) -- (1.875,0.3);
\draw[-] (1.75,0) -- (1.875,0.3);

\draw[-] (0.15,0.24) -- (0.275,0.54);
\draw[-] (0.175,0.18) -- (0.3,0.48);
\draw[-] (0.2,0.12) -- (0.325,0.42);
\draw[-] (0.225,0.06) -- (0.35,0.36);

\draw[-] (0.675,0.3) -- (0.8,0.6);
\draw[-] (0.725,0.3) -- (0.85,0.6);
\draw[-] (0.775,0.3) -- (0.9,0.6);
\draw[-] (0.825,0.3) -- (0.95,0.6);

\draw[-] (-0.075,0.9) -- (0.025,0.66);
\draw[-] (-0.025,0.9) -- (0.05,0.72);
\draw[-] (0.025,0.9) -- (0.075,0.78);
\draw[-] (0.075,0.9) -- (0.1,0.84);

\draw[-] (0.175,0.9) -- (0.275,0.66);
\draw[-] (0.225,0.9) -- (0.3,0.72);
\draw[-] (0.275,0.9) -- (0.325,0.78);
\draw[-] (0.325,0.9) -- (0.35,0.84);

\draw[-] (0.425,0.9) -- (0.525,0.66);
\draw[-] (0.475,0.9) -- (0.55,0.72);
\draw[-] (0.525,0.9) -- (0.575,0.78);
\draw[-] (0.575,0.9) -- (0.6,0.84);

\draw[-] (0.675,0.9) -- (0.775,0.66);
\draw[-] (0.725,0.9) -- (0.8,0.72);
\draw[-] (0.775,0.9) -- (0.825,0.78);
\draw[-] (0.825,0.9) -- (0.85,0.84);

\draw[-] (0.925,0.9) -- (1.025,0.66);
\draw[-] (0.975,0.9) -- (1.05,0.72);
\draw[-] (1.025,0.9) -- (1.075,0.78);
\draw[-] (1.075,0.9) -- (1.1,0.84);

\draw[-] (1.175,0.9) -- (1.275,0.66);
\draw[-] (1.225,0.9) -- (1.3,0.72);
\draw[-] (1.275,0.9) -- (1.325,0.78);
\draw[-] (1.325,0.9) -- (1.35,0.84);

\draw[-] (1.425,0.9) -- (1.525,0.66);
\draw[-] (1.475,0.9) -- (1.55,0.72);
\draw[-] (1.525,0.9) -- (1.575,0.78);
\draw[-] (1.575,0.9) -- (1.6,0.84);

\draw[-] (0.775,0.06) -- (1.225,0.06);
\draw[-] (0.8,0.12) -- (1.2,0.12);
\draw[-] (0.825,0.18) -- (1.175,0.18);
\draw[-] (0.85,0.24) -- (1.15,0.24);
\node at (1,0.18) {$\omega_\bold{p}$};

\draw[ultra thick,->] (0.5,0) -- (0.75,0);
\node[above] at (0.63,0) {$h_{\bold{f}_1}\bold{t}_{\bold{f}_1}$};
\draw[ultra thick,->] (0.75,0) -- (0.625,0.3);
\node at (0.75,0.2) {$h_{\bold{f}_2}\bold{t}_{\bold{f}_2}$};

\node[below] at (0.5,0) {$ \bold{p}_{\bold{f}_1}^1$};
\node[below] at (0.75,0) {$ \bold{p}_{\bold{f}_1}^2$};
\node at (0,0.8) {$K_\bold{{f_1}}$};
\node at (0.25,0.8) {$K_\bold{{f_1}}$};
\node at (0.5,0.8) {$K_\bold{{f_1}}$};
\node at (0.75,0.8) {$K_\bold{{f_1}}$};
\node at (1,0.8) {$K_\bold{{f_1}}$};
\node at (1.25,0.8) {$K_\bold{{f_1}}$};
\node at (1.5,0.8) {$K_\bold{{f_1}}$};

\node at (0.25,0.4) {$N_\bold{{f_1}}$};
\node at (0.78,0.4) {$N_\bold{{f_2}}$};

\draw[->] (0.625,0.6) -- (0.625,0.67);
\draw[->] (0.5625,0.45) -- (0.495,0.422);
\node at (0.68,0.65) {$\bold{f_1}$};
\node at (0.555,0.4) {$\bold{f_2}$};

\fill(0,0) circle(0.5pt);
\fill(0.25,0) circle(0.3pt);
\fill(0.5,0) circle(0.3pt);
\fill(0.75,0) circle(0.3pt);
\fill(1,0) circle(0.3pt);
\fill(1.25,0) circle(0.3pt);
\fill(1.5,0) circle(0.3pt);
\fill(1.75,0) circle(0.5pt);

\fill(-0.125,0.9) circle(0.5pt);
\fill(0.125,0.9) circle(0.3pt);
\fill(0.375,0.9) circle(0.3pt);
\fill(0.625,0.9) circle(0.3pt);
\fill(0.875,0.9) circle(0.3pt);
\fill(1.125,0.9) circle(0.3pt);
\fill(1.375,0.9) circle(0.3pt);
\fill(1.625,0.9) circle(0.5pt);

\node[below] at (1.05,0) {$\bold{p}\in \cP_b^2$};
\node[below] at (1.75,0) {$\bold{p}\in \cP_b^1$};
\node at (1.65,0.1) {$K_\bold{p}$};
\node at (0.9,0.1) {$K^l_\bold{p}$};
\node at (1.15,0.1) {$K^r_\bold{p}$};
\end{tikzpicture}
\caption{An uniform triangulation of $\Om$.}
\label{fig:triangulation}
\end{center}
\end{figure}
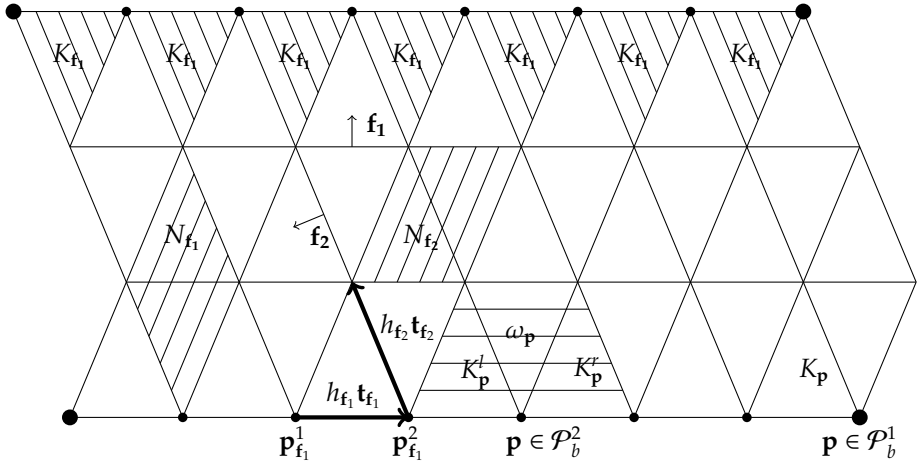

Throughout the paper, a positive constant independent of the mesh size is denoted by $C$, which refers to different values at different places. For ease of presentation, we shall use the symbol $A\lesssim B$ to denote that $A\leq CB$.

\section{Superconvergence for the RT element and the CR element}
In this section, we first improve the superconvergence result for the RT element from a half order by Brandts \cite{Brandts1994Superconvergence} to a full one order. Then, based on this result, we derive a full one order superconvergence result for the CR element, which improves the half order result in \cite{Hu2016Superconvergence}.

\subsection{Second order elliptic problem}
Given $f\in L^2(\Om,\mathbb{R})$, consider a model problem: Seek $\uS\in  H^1_0(\Om,\mathbb{R})$  such that
\be\label{variance}
(\nabla \uS, \nabla v)=(f, v) \text{\quad for any }v\in H^1_0(\Om,\mathbb{R}).
\ee
By introducing an auxiliary variable $\sigS:=\nabla \uS$, the problem can be formulated as the following equivalent mixed problem which finds $(\sigS, \uS)\in H(\text{div},\Om,\mathbb{R}^2)\times L^2(\Om,\mathbb{R})$ such that:
\begin{equation}\label{mixvariance}
\begin{aligned}
(\sigS,\tau)+(\uS,\text{div}\tau)&=0&&\text{ for any }\tau\in H(\text{div},\Om,\mathbb{R}^2),\\
(v,\text{div}\sigS)&=(-f,v)&&\text{ for any }v\in L^2(\Om,\mathbb{R}),
\end{aligned}
\end{equation}
with
$$
H(\text{div},\Om,\mathbb{R}^2) := \{\tau\in L^2(\Om,\mathbb{R}^2),\ \text{div} \tau\in L^2(\Om,\mathbb{R})\}.
$$

The corresponding finite element approximation to \eqref{variance} seeks $u_{\rm CR}\in {\rm CR_0 }(\mathcal{T}_h)$, such that
\be\label{discrete}
(\nabla_h u_{\rm CR},\nabla_h v_h)=(f, v_h)\quad \text{ for any }v_h\in {\rm CR_0 }(\mathcal{T}_h),
\ee
with the CR element spaces \cite{Crouzeix1973Conforming} over $\cT_h$
\begin{equation*}
{\rm CR }(\mathcal{T}_h):=\big \{v\in L^2(\Om,\mathbb{R}): v|_K\in P_1(K)\text{ for any }  K\in\cT_h, \int_e [v]\ds =0\text{ for any }  e\in \cE_h^i\big \},
\end{equation*}
\begin{equation*}
{\rm CR_0 }(\mathcal{T}_h):=\big \{v\in {\rm CR }(\mathcal{T}_h): \int_e v\ds=0\text{ for any }  e\in \cE_h^b\big\}.
\end{equation*}

To analyze the superconvergence of the CR element, we introduce the first order RT element \cite{Raviart1977A}. Its shape function space is
$$
\text{RT}_K:=(P_0(K))^2+ \bold{x}P_0(K)\quad\text{ for any }K\in \cT_h.
$$
The corresponding global finite element space reads
$$
\text{RT}(\cT_h):=\big \{\tau\in H(\text{div},\Om,\mathbb{R}^2): \tau|_K\in \text{RT}_K\text{ for any }K\in \cT_h\big \}.
$$
We use the piecewise constant space to approximate the displacement, namely,
$$
U_{\text{RT}}(\cT_h):=\big \{v\in L^2(\Om,\mathbb{R}):v|_K\in P_0(K) \text{ for any }K\in \cT_h\big \}.
$$
The corresponding RT element method of \eqref{mixvariance} seeks $(\sigRTS,\uRTS)\in \text{RT}(\cT_h)\times U_{\text{RT}}(\cT_h)$ such that
\begin{equation}\label{mixdis}
\begin{aligned}
(\sigRTS,\tau_h)-(\uRTS,\text{div}\tau_h)&=0&&\text{ for any }\tau_h\in \text{RT}(\cT_h),\\
(v_h,\text{div}\sigRTS)&=(f,v_h)&&\text{ for any }v_h\in U_{\text{RT}}(\cT_h).
\end{aligned}
\end{equation}

According to \cite{Brezzi2009On}, the discrete system \eqref{mixdis} has an unique solution $(\sigRTS, \uRTS)\in \text{RT}(\cT_h)\times U_{\text{RT}}(\cT_h)$. Meanwhile, the following optimal error estimates hold with detailed proofs referring to \cite{Douglas1985Global}
\begin{equation*}
\begin{split}
\parallel \sigS-\sigRTS\parallel_{0,\Om}&\lesssim h|\sigS|_{1,\Om},\label{RTu}\\
\parallel \text{div}(\sigS-\sigRTS)\parallel_{0,\Om}&\lesssim h|\sigS|_{2,\Om},
\end{split}
\end{equation*}
provided that $\sigS \in H^2(\Om,\mathbb{R}^2)$.

\subsection{Superconvergence of the RT element}\label{subsec:RT}

In this subsection, we first introduce the analysis in \cite{Brandts1994Superconvergence} for the RT element and then modify the suboptimal error estimate therein for a boundary term to a full one order optimal result.
Introduce the Fortin interpolation operator $\PiRT$, which is widely used in error analysis, such as \cite{Douglas1985Global,Dur1990Superconvergence}. Define $\PiRT:H^1(\Om,\mathbb{R}^2)\rightarrow \rm RT(\cT_h)$ as
\begin{equation*}\label{def:fortin}
\int_e (\PiRT \tau-\tau)^T\bold{n}_e\ds=0\text{\quad for any }e\in \cE_h, \tau\in H^1(\Om,\mathbb{R}^2).
\end{equation*}
It is proved in \cite{Raviart1977A} that for any $\tau \in H^1(\Om,\mathbb{R}^2)$,
\begin{equation}\label{fortinId}
(\text{div}(\tau-\PiRT\bold{\tau}),v_h)=0\quad \text{for any }v_h\in U_{\text{RT}}(\cT_h),
\end{equation}
$$
\parallel \tau- \PiRT\tau\parallel_{0,\Om}\lesssim h|\tau|_{1,\Om}.
$$
It follows from \eqref{mixdis} and \eqref{fortinId} that
\begin{equation*}\label{solinter}
\text{div} \sigRTS= \text{div} \PiRT\sigma.
\end{equation*}
Therefore, $\sigRTS - \PiRT\sigS\in \rm RT(\cT_h)$ is divergence free, and is a piecewise constant vector field. Hence, a substitution of $\tau_h=\sigRTS - \PiRT\sigS$ into \eqref{mixvariance} and \eqref{mixdis} yields
\be\label{sigtosigh}
(\sigS,\sigRTS-\PiRT\sigS)=(\sigRTS,\sigRTS-\PiRT\sigS).
\ee

We need the following result from \cite{Brandts1994Superconvergence} on Sobolev spaces. Denote the subset of the points in $\Om$ having distance less than $h$ from the boundary by $\partial_h \Om$:
$$
\partial_h\Om :=\{\bold{x}\in\Om:\exists \bold{y}\in\partial \Om \text{ such that dist}(\bold{x},\bold{y})\leq h\}.
$$
\begin{Lm}\label{bdtodomain}
For $v\in H^s(\Om,\mathbb{R})$, where $0\leq s\leq \frac{1}{2}$,
$$
\parallel v \parallel_{0,\partial_h \Om}\lesssim h^s\parallel v\parallel_{s,\Om}.
$$
\end{Lm}

Assume that the triangulation $\cT_h$ is uniform. Suppose that the solution of \eqref{mixvariance} satisfies $\sigS\in H^{\frac{5}{2}}(\Om,\mathbb{R}^2)$. Define a matrix $F$, whose transportation  has the two unit normal vectors $\bold{f}_1$ and $\bold{f}_2$ as columns. Denote the canonical basis vectors of $\mathbb{R}^2$ in respectively the $x_1$- and $x_2$-direction by $\bold{e}_1$ and $\bold{e}_2$. In \cite{Brandts1994Superconvergence}, the analysis and    \eqref{sigtosigh} decompose  the error as follows,
\begin{equation*}
\begin{split}
\parallel \sigRTS-\PiRT\sigS\parallel_{0,\Om}^2&=\big (F(\sigRTS-\PiRT\sigS),F^{-T}(\sigS-\PiRT\sigS)\big )\\
&=\sum_{K\in\cT_h} \int_K \big (F(\sigRTS-\PiRT\sigS)\big )^TF^{-T}\big (\sigS-\PiRT\sigS\big )\dx\\
&=\sum_{i,j=1}^2\bold{I}_{ij},
\end{split}
\end{equation*}
where
$$
\bold{I}_{ij}=\sum_{K\in\cT_h} \int_K \bold{e}_i^T F(\sigRTS-\PiRT\sigS)(\bold{e}_i^TF^{-T}\bold{e}_j)(\sigS-\PiRT\sigS)^T\bold{e}_j\dx.
$$
For simplicity, only the sum $\bold{I}_{11}$ is considered. Let $\Om$ be partitioned into parallelograms $N_{\bold{f}_1}$ and the remaining boundary triangles $K_{\bold{f}_1}$. Since $\sigRTS-\PiRT\sigS$ is piecewise constant, the sum $\bold{I}_{11}$ is written as a sum over parallelograms $N_{\bold{f}_1}$ and boundary triangles $K_{\bold{f}_1}$ in \cite[Thm.~3.2]{Brandts1994Superconvergence}:
\be\label{totalterm}
|\bold{I}_{11}|\leq |\bold{I}_{11}^1|+|\bold{I}_{11}^2|,
\ee
where
$$
\bold{I}_{11}^1=(\bold{e}_1^TF^{-T}\bold{e}_1)\sum_{N_{\bold{f}_1}}\int_{N_{\bold{f}_1}}\bold{e}_1^T F(\sigRTS-\PiRT\sigS)(\sigS-\PiRT\sigS)^T\bold{e}_1\dx,
$$
\begin{equation}\label{I2original}
\bold{I}_{11}^2=(\bold{e}_1^TF^{-T}\bold{e}_1)\sum_{K_{\bold{f}_1}}\bold{e}_1^T F(\sigRTS-\PiRT\sigS)\int_{K_{\bold{f}_1}} (\sigS-\PiRT\sigS)^T\bold{e}_1\dx.
\end{equation}
Note that $\bold{e}_1^TF(\sigRTS-\PiRT\sigS)=(\sigRTS-\PiRT\sigS)^T\bold{f}_1$ is the normal component of $\sigRTS-\PiRT\sigS$ to the shared edge of the two triangles forming a parallelogram $N_{\bold{f}_1}$. Thus, $\bold{e}_1^TF(\sigRTS-\PiRT\sigS)$ is constant on $N_{\bold{f}_1}$, and therefore, leads to the following superconvergence property \cite{Brandts1994Superconvergence}:
\be\label{1termfinal}
|\bold{I}_{11}^1|\lesssim h^2\parallel \sigRTS-\PiRT\sigS\parallel_{0,\Om}|\sigS|_{2,\Om}.
\ee
For the sum $\bold{I}_{11}^2$ of boundary terms, the analysis in \cite{Brandts1994Superconvergence} showed
\be\label{2termI11}
|\bold{I}_{11}^2|\lesssim h^{3/2}\parallel \sigS-\PiRT\sigS\parallel_{0,\Om}|\sigS|_{\frac{3}{2},\Om}.
\ee
The estimate \eqref{2termI11} is resulted from a direct application of Lemma \ref{bdtodomain}. Since the estimate in Lemma \ref{bdtodomain} can not be improved as shown by a counter example in \cite{J1972Non}, it is very difficult to improve the factor in \eqref{2termI11} from $h^{3/2}$ to $h^2$ following that analysis.

A new analysis for a full one order superconvergence result of the RT element is provided in the following. The main idea here is to employ a discrete Helmholtz decomposition \cite{Arnold,Brezzi1985}  of $\sigRTS-\PiRT\sigS$ in the following lemma. As a result, it allows for some vital cancellation between the boundary terms in $\bold{I}_{11}^2$ sharing a common vertex. This idea was also employed in \cite{li2017global} to prove  a full one order optimal superconvergence result of the RT element.
\begin{Lm}\label{lemma2}
For any function $\bold{\tau}_h\in \rm RT(\cT_h)$ which satisfies
$
\text{div } \bold{\tau}_h=0,
$
$$
\bold{\tau}_h\in \bold{curl} \conPone,
$$
with $\conPone:=\{v\in H^1(\Omega,\Rmath):v|_K\in P_1(K,\Rmath)\text{ for any }K\in\Tma\}$.
\end{Lm}
We need the following result for the interpolation operator $\PiRT$.
\begin{Lm}\label{linearsigma}
For any $\bold{p}\in\cP_b^2$, $K_{\bold{p}}^l$, $K_{\bold{p}}^r\in K_{\bold{f}_i}$, $i=1,\ 2$, if $\tau$ is linear on the patch $\omega_{\bold{p}}$, then
\begin{equation*}
\begin{split}
\int_{K_{\bold{p}}^l}(\tau-\PiRT\tau)\dx= \int_{K_{\bold{p}}^r}(\tau-\PiRT\tau)\dx.
\end{split}
\end{equation*}
\end{Lm}
\begin{proof}
Denote the centroid, the vertices and the edges of element $K_{\bold{p}}^l$ by $\bold{M}_{K_{\bold{p}}^l }$, $\{\bold{p}_i^l\}_{i=1}^3$ and $\{e_i^l\}_{i=1}^3$, and those of element $K_{\bold{p}}^r$ by $\bold{M}_{K_{\bold{p}}^r}$, $\{\bold{p}_i^r\}_{i=1}^3$ and $\{e_i^r\}_{i=1}^3$. For edge $e_i^l$, denote the midpoint, the unit outward normal vector and the perpendicular height by $\bold{m}_i^l$, $\bold{n}^l_i$ and $d^l_i$, respectively. And denote those of edge $e_i^r$ by $\bold{m}_i^r$, $\bold{n}^r_i$ and $d^r_i$, respectively. The basis functions of the RT element on elements $K_{\bold{p}}^l$ and $K_{\bold{p}}^r$ are denoted by $\phi_i^l=\frac{1}{d_i^l}(\bold{x}-\bold{p}_i^l)$ and $\phi_i^r=\frac{1}{d_i^r}(\bold{x}-\bold{p}_i^r),\ 1\leq i\leq 3$, respectively.

Since $\tau$ is linear on the patch $\omega_{\bold{p}}$, $
\tau(\bold{x})= \tau(\bold{p}) + \nabla \tau\cdot (\bold{x} - \bold{p}).
$ Thus,
$$
\tau(\bold{x})-\PiRT\tau(\bold{x}) = (I - \Pi_{\rm RT})\big (\nabla \tau\cdot(\bold{x}-\bold{p})\big ).
$$
The fact that
$$
\int_{K_{\bold{p}}^l } (I - \PiRT)\big (\nabla \tau\cdot( \bold{M}_{K_{\bold{p}}^l }-\bold{p})\big )\dx =0\ \text{ and } \int_{K_{\bold{p}}^l } \nabla \tau\cdot(\bold{x}- \bold{M}_{K_{\bold{p}}^l })\dx =0
$$
leads to
\be\label{onlyPiRT}
\int_{K_{\bold{p}}^l} \big (\tau(\bold{x})-\PiRT\tau(\bold{x})\big )\dx = -\int_{K_{\bold{p}}^l} \PiRT\big (\nabla \tau\cdot(\bold{x}-\bold{M}_{K_{\bold{p}}^l })\big )\dx.
\ee
Note that $\nabla\tau|_{K_{\bold{p}}^l }=\nabla \tau|_{K_{\bold{p}}^r}$, $\bold{n}^l_i=\bold{n}^r_i$ and $\bold{m}_i^l-\bold{M}_{K_{\bold{p}}^l}=\bold{m}_i^r-\bold{M}_{K_{\bold{p}}^r }$. Thus
$$
(\nabla \tau\cdot(\bold{m}_i^l-\bold{M}_{K_{\bold{p}}^l }))^T\bold{n}_i^l=(\nabla \tau\cdot(\bold{m}_i^r-\bold{M}_{K_{\bold{p}}^r }))^T\bold{n}_i^r.
$$
Since $\int_{K_{\bold{p}}^l} \phi_i^l\dx = \int_{K_{\bold{p}}^r} \phi_i^r\dx$, this and \eqref{onlyPiRT} lead to
$$
\int_{K_{\bold{p}}^l}(\tau-\PiRT\tau)\dx= \int_{K_{\bold{p}}^r}(\tau-\PiRT\tau)\dx,
$$
which completes the proof.
\end{proof}

Employing the discrete Helmholtz decomposition, we can improve the estimate of the term $ \bold{I}_{11}$ of \cite{Brandts1994Superconvergence} in the following lemma.
\begin{Lm}\label{lm:RTbdest}
Suppose that $(\sigS, \uS)$ denotes the solution to \eqref{mixvariance} with $\sigS\in H^{\frac{5}{2}}(\Om,\mathbb{R}^2)$, $(\sigRTS, \uRTS)$ denotes the solution to \eqref{mixdis} on a uniform triangulation $\cT_h$. It holds that
\begin{equation*} 
|\bold{I}_{11}|\lesssim h^2 (| \sigS|_{\frac{5}{2},\Om}+\kappa |\ln h|^{1/2}|\sigS|_{1,\infty,\Om})\parallel \sigRTS-\PiRT \sigS \parallel_{0,\Om}.
\end{equation*}
\end{Lm}
\begin{proof}
By Lemma \ref{lemma2}, there exists $w_h\in \cP_1$ such that
$$\sigRTS-\PiRT\sigS=\bold{curl} w_h \in \big(U_{\rm RT}(\mathcal{T}_h)\big)^2.$$
Then, the term $\bold{I}_{11}^2$ in \eqref{I2original} reads
\be\label{2term1}
\bold{I}_{11}^2=(\bold{e}_1^TF^{-T}\bold{e}_1)\sum_{K_{\bold{f}_1}} \bold{e}_1^TF\bold{curl} w_h\int_{K_{\bold{f}_1}}(\sigS-\PiRT \sigS)^T\bold{e}_1\dx.
\ee
Since
\begin{equation}\label{2term2}
\bold{e}_1^TF\bold{curl}w_h=\frac{1}{h_{\bold{f}_1}}\int_{e_{\bold{f}_1}}\nabla w_h\cdot \bold{t}_{\bold{f}_1}\ds=\frac{w_h(\bold{p}_{\bold{f}_1}^2)-w_h(\bold{p}_{\bold{f}_1}^1)}{h_{\bold{f}_1}},
\end{equation}
a substitution of \eqref{2term2} into \eqref{2term1} leads to
\be\label{2term3}
\begin{split}
|\bold{I}_{11}^2|\lesssim &\big|\sum_{\bold{p}\in\cP_b^2} \frac{w_h(\bold{p})}{h_{\bold{f}_1}}\big(\int_{K_{\bold{p}}^l}(\sigS-\PiRT\sigS)^T\bold{e}_1\dx- \int_{K_{\bold{p}}^r}(\sigS-\PiRT\sigS)^T\bold{e}_1\dx\big)\big|\\
& + \sum_{\bold{p}\in\cP_b^1} \big| \frac{w_h(\bold{p})}{h_{\bold{f}_1}}\int_{K_{\bold{p}}}(\sigS-\PiRT\sigS)^T\bold{e}_1\dx\big|.
\end{split}
\ee
Lemma \ref{linearsigma} and the Bramble-Hilbert lemma show
\begin{equation}\label{2term4}
\begin{split}
\big |\int_{K_{\bold{p}}^l}(\sigS-\PiRT\sigS)^T\bold{e}_1\dx- \int_{K_{\bold{p}}^r}(\sigS-\PiRT\sigS)^T\bold{e}_1\dx\big |\lesssim h^3|\sigS|_{2,\omega_{\bold{p}}}.
\end{split}
\end{equation}
A substitution of \eqref{2term4} and the Cauchy-Schwarz inequality into \eqref{2term3} yields
\begin{equation*}
\begin{split}
|\bold{I}_{11}^2|\lesssim &h\big(\sum_{\bold{p}\in\cP_b^2} \parallel w_h\parallel_{0,\omega_{\bold{p}}}^2\big)^{1/2}\big(\sum_{\bold{p}\in\cP_b^2}|\sigS|_{2,\omega_{\bold{p}}}^2\big)^{1/2} + h^2\big(\sum_{\bold{p}\in\cP_b^1} \parallel w_h\parallel_{0,\infty,K_{\bold{p}}}^2\big)^{1/2}\big(\sum_{\bold{p}\in\cP_b^1}|\sigS|_{1,\infty,K_{\bold{p}}}^2\big)^{1/2} \\
\lesssim &h \parallel w_h\parallel_{0,\partial_h \Om} | \sigS|_{2,\partial_h\Om}+\kappa h^2 |\sigS|_{1,\infty,\Om} \parallel w_h\parallel_{0,\infty,h}.
\end{split}
\end{equation*}
Lemma \ref{bdtodomain} implies that
\begin{equation}\label{eq:ess}
h \parallel w_h\parallel_{0,\partial_h \Om} | \sigS|_{2,\partial_h\Om}\lesssim h^2 \parallel w_h\parallel_{\frac{1}{2},\Om} |\sigS|_{\frac{5}{2},\Om}\lesssim h^2 \parallel w_h\parallel_{1, \Om} |\sigS|_{\frac{5}{2},\Om}.
\end{equation}
A discrete Sobolev inequality holds from \cite{Bramble1986TheCO,brenner2007mathematical} that
\begin{align}\label{eq:discreteSobolev}
\parallel w_h\parallel_{0,\infty,h}\lesssim |\ln h|^{1/2}\parallel w_h\parallel_{1,\Om}.
\end{align}
Then
\be\label{2termfinal}
|\bold{I}_{11}^2|\lesssim (h^2 |\sigS|_{\frac{5}{2},\Om}+\kappa h^2|\ln h|^{1/2}|\sigS|_{1,\infty,\Om}) \parallel w_h\parallel_{1,\Om}.
\ee
A substitution of \eqref{1termfinal} and \eqref{2termfinal} into \eqref{totalterm} concludes
$$
|\bold{I}_{11}|\lesssim (h^2 |\sigS|_{\frac{5}{2},\Om}+\kappa h^2|\ln h|^{1/2}|\sigS|_{1,\infty,\Om})\parallel \sigRTS-\PiRT\sigS\parallel_{0,\Om},
$$
which completes the proof.
\end{proof}

Similar arguments for the sums $ \bold{I}_{12}$, $ \bold{I}_{21}$ and $ \bold{I}_{22}$ prove a full one order superconvergence property for the RT element as follows.
\begin{Th}\label{Lm:bdRT}
Suppose that $(\sigS, \uS)$ is the solution to \eqref{mixvariance} with $\sigS\in H^{\frac{5}{2}}(\Om,\mathbb{R}^2)$, and $(\sigRTS, \uRTS)$ is the solution to \eqref{mixdis} on a uniform triangulation $\cT_h$. It holds that
$$
\parallel \sigRTS-\PiRT\sigS\parallel_{0,\Om}\lesssim h^2 \big (| \sigS|_{\frac{5}{2},\Om}+\kappa |\ln h|^{1/2}|\sigS|_{1,\infty,\Om}\big ).
$$
\end{Th}
\begin{remark}
Lemma~\ref{lm:RTbdest} employs Lemma~\ref{bdtodomain} to control \eqref{eq:ess} instead of   using the infinity norm in \cite[Lemma~3.7]{li2017global}. This avoids  to control the number of vertices on the edges and allows  a weaker assumption   \eqref{eq:meshcon} on $(\alpha,\sigma)$-meshes in Subsect.~4.3.
\end{remark}
\subsection{Superconvergence of the CR element}\label{sec:Kdefine}
A full one order superconvergence result for the CR element follows from a special relation between the RT element and the CR element.

A post-processing mechanism \cite{Brandts1994Superconvergence} is employed in \cite{Hu2016Superconvergence} for the superconvergence analysis of the CR element. Given $\textbf{q}\in \rm RT(\cT_h)$, define function $K_h \textbf{q}\in {\rm CR }(\mathcal{T}_h)\times {\rm CR }(\mathcal{T}_h)$ as follows.
\begin{Def}\label{Def:R}
1.For each interior edge $e\in\cE_h^i$, the elements $K_e^1$ and $K_e^2$ are the pair of elements sharing $e$. Then the value of $K_h \textbf{q}$ at the midpoint $\textbf{m}_e$ of $e$ is
$$
K_h \textbf{q}(\textbf{m}_e)=\frac{1}{2}\big(\textbf{q}|_{K_e^1}(\textbf{m}_e)+\textbf{q}|_{K_e^2}(\textbf{m}_e)\big).
$$

2.For each boundary edge $e\in\cE_h^b$, let $K$ be the element having $e$ as an edge, and $K'$ be an element sharing an edge $e'\in\cE_h^i$ with $K$. Let $e''$ denote the edge of $K'$ that does not intersect with $e$, and $\textbf{m}$, $\textbf{m}'$ and $\textbf{m}''$ denote the midpoints of the edges $e$, $e'$ and $e''$, respectively. Then the value of $K_h \textbf{q}$ at the point $\textbf{m}$ is
$$
K_h \textbf{q}(\textbf{m})=2K_h \textbf{q}(\textbf{m}')-K_h \textbf{q}(\textbf{m}'').
$$
\begin{center}
\begin{tikzpicture}[xscale=2.5,yscale=2.5]
\draw[-] (-0.5,0) -- (2,0);
\draw[-] (0,0) -- (0.5,1);
\draw[-] (0.5,1) -- (2,1);
\draw[-] (0.5,1) -- (1.5,0);
\draw[-] (2,1) -- (1.5,0);
\node[below, right] at (1,0.5) {\textbf{m}'};
\node[above] at (1.25,1) {\textbf{m}''};
\node[below] at (0.75,0) {\textbf{m}};
\node at (0.7,0.4) {K};
\node at (1.4,0.75) {K'};
\node at (1,0) {e};
\node at (1.4,0.2) {e'};
\node at (1.7,1) {e''};
\node at (0.3,-0.1) {$\partial \Om $};
\fill(1,0.5) circle(0.5pt);
\fill(1.25,1) circle(0.5pt);
\fill(0.75,0) circle(0.5pt);
\end{tikzpicture}
\end{center}
\end{Def}

Due to the superconvergence result of the RT element in Theorem \ref{Lm:bdRT} and the special relation between the RT element and the CR element \cite{Marini1985An}, the superconvergence result of the CR element for \eqref{variance} can be improved from a half order to a full one order following the analysis in \cite{Hu2016Superconvergence}.

\begin{Th}\label{superbd}
Suppose that $\uS\in H^{\frac{7}{2}}(\Om,\mathbb{R})\cap H^1_0(\Om,\mathbb{R})$ is the solution to \eqref{variance}, $u_{\rm CR}$ is the solution to \eqref{discrete} by the CR element on a  uniform triangulation $\cT_h$, and $f\in W^{1,\infty}(\Om,\mathbb{R})$. It holds that
$$
\parallel \nabla \uS-K_h \nabla_h u_{\rm CR}\parallel_{0,\Om}\lesssim h^2(|\uS|_{\frac{7}{2},\Om}+\kappa |\ln h|^{1/2}|\uS|_{2,\infty,\Om}+|f|_{1,\infty,\Om}).
$$
\end{Th}
\begin{remark}
As analyzed in \cite{Brandts1994Superconvergence}, the vector $K_h \PiRT\sigS$ is a higher order approximation of $\sigS$ than $\PiRT\sigS$ itself. Thanks to the one order superconvergence  result of the RT element in Theorem \ref{superbd} and the equivalence between the RT element and the ECR element \cite{Hu2015The}, a similar argument may  prove a full one order superconvergence result  for the ECR element method of the Poisson problem \eqref{variance}.
\end{remark}
\section{Superconvergence for the HHJ element and the Morley element}
\label{sec:HHJ}
Given $f\in L^2(\Omega,\Rmath)$, the plate bending model problem finds $\ubi\in H^2_0(\Omega,\Rmath)$ such that
\begin{equation}\label{Plateequation}
  (\nabla^2 \ubi,\nabla^2v)=(f,v)\quad\text{for any }v\in H^2_0(\Omega,\Rmath).
\end{equation}
Suppose $\Smath:=\text{symmetric }\Rmath^{2\times2}$. Given $K\in\Tma$ and $\tau\in H^1(K,\Smath)$, let $$\tau_{\nBi\nBi}=\nBi^T\tau\nBi$$
with the unit outnormal  $\nBi$ of $\partial K$.
Define the following two spaces
 \begin{equation*}
  \begin{split}
  \sSpaMBi: =&\{\bm{\tau}\in L^2(\Omega,\Smath):\tau|_K\in H^1(K,\Smath)\text{ for any }K\in\Tma,\\
  &\text{ and }\bm{\tau}_{\nBi\nBi}\text{ is continuous across interior edges}\},\\
 \uSpaMBi: =&\{v\in H^1_0(\Omega,\Rmath):v|_K\in H^2(K,\Rmath)\text{ for any }K\in\Tma\}.
  \end{split}
 \end{equation*}
For any $\tau\in \sSpaMBi$ and $v\in \uSpaMBi$, define the bilinear form
 \begin{align*}
 \langle{\rm div}{\rm \bm{div}}_h\tau,v\rangle :=-\sum_{K\in\Tma}\Big((\tau,\nabla^2v)_{L^2(K)}-\int_{\partial K}\tau_{\nBi\nBi}\frac{\partial v}{\partial \nBi}ds\Big).
 \end{align*}
By introducing an auxiliary variable $\sbiM:=\nabla^2\ubi$, the mixed formulation of \eqref{Plateequation} seeks $(\sbiM,\ubiM)\in \sSpaMBi\times \uSpaMBi$, see \cite{johnson1973convergence},
 \begin{equation}\label{MixedPlate}
\begin{split}
 &(\sbiM, \tau)+\langle{\rm div}{\rm \bm{div}}_h\tau,\ubiM\rangle =0\quad \text{ for any }\tau\in \sSpaMBi,\\
 &\langle{\rm div}{\rm \bm{div}}_h\sbiM,v\rangle=(-f, v)\quad \text{ for any }v\in \uSpaMBi.
\end{split}
 \end{equation}

The Morley element method of \eqref{Plateequation} finds $\uMorley\in {\rm M}(\mathcal{T}_h)$ such that
\begin{equation}\label{PlateDiscrete}
  (\nabla^2_{h}\uMorley,\nabla^2_{h}v)=(f,v)\quad\text{for any }v\in {\rm M}(\mathcal{T}_h),
\end{equation}
where the Morley element space is defined in \cite{morley1968triangular} by
\begin{equation*}
\begin{split}
{\rm M}(\mathcal{T}_h): =&\{v\in L^2(\Om, \mathbb{R}):v|_K\in P_2(K)  \text{ for each }K\in\cT_h, v \text{ is continuous at each}\\
&\text{ interior vertex and vanishes at each boundary vertex}, \int_e [\frac{\partial v}{\partial \nBi}]\ds=0\\
& \text{ for all }e\in \cE_h^i, \text{ and }\int_e \frac{\partial v}{\partial \nBi}\ds=0 \text{ for all }e\in \cE_h^b \}.
\end{split}
\end{equation*}

Introduce the first order HHJ element \cite{johnson1973convergence}:
  \begin{equation*}
  \begin{split}
  \SHHJ: =&\{\tau\in\sSpaMBi:\tau|_K\in P_0(K,\Smath)\text{ for any }K\in\Tma\},\\
\UHHJ: =&\{v\in H^1_0(\Omega,\Rmath):v|_K\in P_1(K,\Rmath)\text{ for any }K\in\Tma\}.
  \end{split}
 \end{equation*}
The corresponding approximation to  \eqref{MixedPlate} finds $(\sHHJ,\uHHJ)\in \SHHJ\times  \UHHJ$ such that
 \begin{equation}\label{MixedplateDiscrete}
\begin{split}
 &(\sHHJ, \tau)+\langle{\rm div}{\rm \bm{div}}_h\tau,\uHHJ\rangle =0\quad \text{ for any }\tau\in \SHHJ,\\
 &\langle{\rm div}{\rm \bm{div}}_h\sHHJ,v\rangle=(-f, v)\quad \text{ for any }v\in  \UHHJ.
\end{split}
 \end{equation}

\subsection{Superconvergencec of the HHJ element}
Introduce the interpolation  operator $\PiHHJ:\sSpaMBi\rightarrow \SHHJ$  \cite{a1976mixed}:
 \begin{equation}\label{interDefHHJ}
  \int_e(\PiHHJ\tau)_{\nBi\nBi}ds=\int_e\tau_{\nBi\nBi}ds\quad\text{for all } e\in\mathcal{E}_h.
 \end{equation}
 Moreover if $\tau\in H^1(\Omega,\Smath)$, then
 \begin{equation}\label{interpolationHHJ}
  \|\tau-\PiHHJ\tau\|_{0,\Omega}\lesssim h|\tau|_{1,\Omega}.
 \end{equation}
Since $v|_K\in P_1(K,\Rmath)$ for any $v\in\UHHJ$ and $K\in\Tma$, it holds that
\begin{align}
\label{interpolationHHJtwo}
\langle{\rm div}{\rm \bm{div}}_h(\tau-\PiHHJ\tau),v\rangle=0\text{ for any }v\in \UHHJ.
\end{align}

Define the rigid motion space
\[
{\rm RM}=\left\{\begin{pmatrix}c_1-c_3x_2\\
c_2+c_3x_1\end{pmatrix}\bigg| c_1,c_2,c_3\in\Rmath\right\}.
\]

The subsequent parts analyze the superconvergence of the HHJ element. The argument is similar as in Section \ref{subsec:RT}.
As proved in \cite[Lemma 5.1]{Hu2016Superconvergence}, it holds
\begin{equation*}
 (\sHHJ-\sbiM,\sHHJ-\PiHHJ\sbiM)=0.
\end{equation*}
This leads to
  \begin{equation*}
   (\sHHJ-\PiHHJ\sbiM,\sHHJ-\PiHHJ\sbiM)=(\sHHJ-\PiHHJ\sbiM,\sbiM-\PiHHJ\sbiM).
  \end{equation*}
  Let $\Psi_{i}\in P_0(K,\Smath),1\leq i\leq 3$ denote the basis functions, i.e., $(\Psi_i)_{\bfv_j\bfv_j}=\delta_{ij}$, where $\{\bfv_i\}_{i=1}^3$ are the normal vectors as shown in Figure \ref{fig:triangulation}. Then, the following decomposition holds:
  \begin{equation}\label{eqm}
   (\sHHJ-\PiHHJ\sbiM,\sHHJ-\PiHHJ\sbiM)=\sum_{i=1}^3\bold{J}_i,
  \end{equation}
where
  \begin{equation*}
    \bold{J}_i:=\sum_{K\in\Tma}\int_K(\sHHJ-\PiHHJ\sbiM)_{\bfv_i\bfv_i}\Psi_{i}:(\sbiM-\PiHHJ\sbiM)dx.
  \end{equation*}
 For simplicity, only the sum $\bold{J}_1$ is considered in \cite{Hu2016Superconvergence}. Since $(\sHHJ-\PiHHJ\sbiM)_{\bfv_1\bfv_1}$ is continuous and constant on $N_{\bfv_1}$, and $\Psi_{1}$ is constant on $N_{\bfv_1}$, the sum $\bold{J}_1$ is rewritten as a sum over parallelogram $N_{\bfv_1}$ and boundary triangles $K_{\bfv_1}$ in \cite[Thm.~5.3]{Hu2016Superconvergence}:
  \begin{align*}
  |\bold{J}_1|\leq |\bold{J}_1^1|+|\bold{J}_1^2|,
  \end{align*}
  where
    \begin{equation*}
   \bold{J}_1^1=\sum_{N_{f_1}}(\sHHJ-\PiHHJ\sbiM)_{\bfv_1\bfv_1}\int_{N_{f_1}}\Psi_{1}:(\sbiM-\PiHHJ\sbiM)dx,
   \end{equation*}
\begin{equation} \label{itemtotaltwo}
   \bold{J}^2_1=\sum_{K_{f_1}}\int_{K_{f_1}}(\sHHJ-\PiHHJ\sbiM)_{\bfv_1\bfv_1}\Psi_{1}:(\sbiM-\PiHHJ\sbiM)dx.
  \end{equation}
Theorem~5.3 of \cite{Hu2016Superconvergence} has an optimal superconvergence result for $\bold{J}_1^1$
  \begin{align}
  \label{HHJest1}
   |\bold{J}_1^1|\lesssim h^2\|\sHHJ-\PiHHJ\sbiM\|_{0,\Omega}\|\sbiM\|_{2,\Omega},
   \end{align}
    and a suboptimal result for $\bold{J}_1^2$  
   \begin{align*}
      |\bold{J}_1^2|\lesssim h^{\frac{3}{2}}\|\sHHJ-\PiHHJ\sbiM\|_{0,\Omega}\|\sbiM\|_{\frac{3}{2},\Omega}.
  \end{align*}
 We will modify the estimate for  $\bold{J}_1^2$ in \cite{Hu2016Superconvergence} to an optimal superconvergence result by use of  the following two lemmas.
\begin{Lm}\cite[Theorem 5.2]{Krendl2016decomposition}\label{lemm:hel}
Let $\Omega$ be simply connected. Given $\tau_h\in \SHHJ$, if $\langle{\rm div}{\rm \bm{div}}_h \tau_h,v_h\rangle =0$ for any $v_h\in\UHHJ $, then there exists a unique function $\phi_h\in (\conPone)^2/{\rm RM}$ such that
\begin{align*}
\tau_h=H^T\epsilon(\phi_h)H
\end{align*}
with $H=\begin{pmatrix}0 & -1\\1 & 0\end{pmatrix}$ and $\epsilon(\phi_h)=\big(\nabla\phi_h+(\nabla\phi_h)^T\big)/2$.
\end{Lm}
It follows from \eqref{MixedPlate}, \eqref{MixedplateDiscrete} and \eqref{interpolationHHJtwo} that
\begin{align*}
\langle{\rm div}{\rm \bm{div}}_h(\sHHJ-\PiHHJ\sbiM), v_h\rangle =0 \text{ for any } v_h\in\UHHJ.
\end{align*}
Lemma~\ref{lemm:hel} shows that there exists $\phi_h\in(\conPone)^2/{\rm RM}$ such that
\begin{equation}\label{eq:prokey}
\sHHJ-\PiHHJ\sbiM=H^T\epsilon(\phi_h)H.
\end{equation}
\begin{Lm}\label{lem:highHHJ}
For any $\textbf{p}=(p_1, p_2)\in \Pma_b^2,K_{\textbf{p}}^l,K_{\textbf{p}}^r\in K_{\bfv_1}$, if $\tau_\Pma \in \sSpaMBi$ is linear on the patch $\omega_{\textbf{p}}$, then
\[
\int_{K_{\textbf{p}}^l}(\tau_\Pma-\PiHHJ\tau_\Pma)\,dx=\int_{K_{\textbf{p}}^r}(\tau_\Pma-\PiHHJ\tau_\Pma)\,dx,
\]
where $ \Pma_b^2$, $K_{\textbf{p}}^l$ , $K_{\textbf{p}}^r$ and $ K_{\bfv_1}$ are defined in Section \ref{subsec:RT} as shown in Figure \ref{fig:triangulation}.
\end{Lm}
\begin{proof} Recall some notations used in Lemma \ref{linearsigma}, namely,  the centroid $\bold{M}_{K_{\bold{p}}^l }=(M_{K_{\bold{p}}^l }^1, M_{K_{\bold{p}}^l }^2)$, the vertices $\{\bold{p}_i^l\}_{i=1}^3$,  and the edges $\{e_i^l\}_{i=1}^3$ of element $K_{\bold{p}}^l$,  and those of element $K_{\bold{p}}^r$ by $\bold{M}_{K_{\bold{p}}^r}=(M_{K_{\bold{p}}^r }^1, M_{K_{\bold{p}}^r }^2)$, $\{\bold{p}_i^r\}_{i=1}^3$ and $\{e_i^r\}_{i=1}^3$. For edge $e_i^l$, denote the midpoint, the unit outward normal vector and the tangent vector by $\bold{m}_i^l$, $\bold{n}^l_i$ and $\bold{t}^l_i$, respectively. And denote those of edge $e_i^r$ by $\bold{m}_i^r$, $\bold{n}^r_i$ and $\bold{t}^r_i$, respectively. The basis functions of the HHJ element on elements $K_{\bold{p}}^l$ and $K_{\bold{p}}^r$ are denoted by
$$
\Psi_i^l=\frac{1}{2\big((n_i^l)^T\bold{t}^l_{i-1}\big)\big((n_i^l)^T\bold{t}^l_{i+1}\big)}\big(\bold{t}^l_{i+1}(\bold{t}^l_{i-1})^T + \bold{t}^l_{i-1}(\bold{t}^l_{i+1})^T\big)
$$
and
$$
\Psi_i^r=\frac{1}{2\big((n_i^r)^T\bold{t}^r_{i-1}\big)\big((n_i^r)^T\bold{t}^r_{i+1}\big)}\big(\bold{t}^r_{i+1}(\bold{t}^r_{i-1})^T + \bold{t}^r_{i-1}(\bold{t}^r_{i+1})^T\big),
$$
respectively. Note that $(\Psi_i^l)_{\bold{n}^l_j\bold{n}^l_j}=\delta_{ij}$ and $(\Psi_i^r)_{\bold{n}^r_j\bold{n}^r_j}=\delta_{ij}$.

Since $\tau_\Pma=(\tau_\Pma^{ij})_{i,j=1}^2$ is linear on the patch $\omega_{\bold{p}}$,
$$
\tau_\Pma(\bold{x})= \tau_\Pma(\bold{p}) + (x_1 - p_1)H_1 + (x_2 - p_2)H_2.
$$
where $H_k^{ij}=\frac{\partial}{\partial_{x_k}}\tau_\Pma^{ij}$, $i, j, k=1, 2$,  and $H_k=(H_k^{ij})_{i,j=1}^2$ are constant matrices.
Thus,
$$
\tau_\Pma(\bold{x})-\PiHHJ\tau_\Pma(\bold{x}) = (I - \PiHHJ)\big ((x_1 - M_{K_{\bold{p}}^l }^1)H_1 + (x_2 - M_{K_{\bold{p}}^l }^2)H_2\big ).
$$
The fact that
$$
\int_{K_{\bold{p}}^l }  (\bold{x}- \bold{M}_{K_{\bold{p}}^l } )\dx =\bold{0}
$$
leads to
\begin{equation}\label{onlyPiHHJ}
\int_{K_{\bold{p}}^l} \big (\tau_\Pma(\bold{x})-\PiHHJ\tau_\Pma(\bold{x})\big )\dx = -\int_{K_{\bold{p}}^l} \PiHHJ\big ((x_1 - M_{K_{\bold{p}}^l }^1)H_1 + (x_2 - M_{K_{\bold{p}}^l }^2)H_2\big )\dx.
\end{equation}
Note that $H_k|_{K_{\bold{p}}^l }=H_k|_{K_{\bold{p}}^r}$, $k=1, 2$ and  $\bold{n}^l_i=\bold{n}^r_i$, thus,
$$
(\bold{n}_i^l)^TH_k\bold{n}_i^l = (\bold{n}_i^r)^TH_k\bold{n}_i^r.
$$
Since $\int_{K_{\bold{p}}^l}  \Psi_i^l\dx = \int_{K_{\bold{p}}^r} \Psi_i^r\dx$ and $\bold{m}_i^l-\bold{M}_{K_{\bold{p}}^l}=\bold{m}_i^r-\bold{M}_{K_{\bold{p}}^r }$, these and \eqref{onlyPiHHJ} lead to
$$
\int_{K_{\bold{p}}^l}(\tau_\Pma-\PiHHJ\tau_\Pma)\dx= \int_{K_{\bold{p}}^r}(\tau_\Pma-\PiHHJ\tau_\Pma)\dx,
$$
which completes the proof.
\end{proof}
  \begin{Lm}\label{lem:Jonetwo}Assume $\sbiM\in H^{\frac{5}{2}}(\Omega)$. Then, 
  \begin{align*}
  |\bold{J}_1^2|\lesssim h^2(|\sbiM|_{\frac{5}{2},\Omega}+\kappa|\ln h|^{1/2}|\sbiM|_{1,\infty,\Omega})\|\sHHJ-\PiHHJ\sbiM\|_{0,\Omega}.
  \end{align*}
  \end{Lm}
  \begin{proof}
  Notations $\bfv_1$, $\btv_{\bfv_1}$, $h_{\bfv_1}$, $\textbf{p}_{\btv_{\bfv_1}}^1$, $\textbf{p}_{\btv_{\bfv_1}}^2$, $K_{\bfv_1}$, $ \Pma_b^1$, $ \Pma_b^2$, $K_{\textbf{p}}^l$ , $K_{\textbf{p}}^r$ and $ K_{\bfv_1}$ below are defined in Section \ref{subsec:RT} as shown in Figure \ref{fig:triangulation}.
  A substitution of \eqref{eq:prokey} into \eqref{itemtotaltwo} shows that
  \begin{align}\label{eq:Jonetwofirst}
   \bold{J}^2_1=&\sum_{K_{\bfv_1}}(H^T\epsilon(\phi_h)H)_{\bfv_1\bfv_1}\int_{K_{\bfv_1}}\Psi_{1}:(\sbiM-\PiHHJ\sbiM)\,dx.
  \end{align}
  Since
  \begin{align}\label{eq:Jonetwosec}
 (H^T\epsilon(\phi_h)H)_{\bfv_i\bfv_i}=\frac{1}{h_{\bfv_1}}\int_{e_{\bfv_1}}\nabla(\btv_{\bfv_1}\cdot\phi_h)\cdot \btv_{\bfv_1}\,dx=\frac{\btv_{\bfv_1}\cdot\phi_h(\textbf{p}_{\btv_{\bfv_1}}^2)-\btv_{\bfv_1}\cdot\phi_h(\textbf{p}_{\btv_{\bfv_1}}^1)}{h_{{\bfv_1}}},
  \end{align}
  this leads to
   \begin{align}\label{esHHJ2}
   \begin{aligned}
   |\bold{J}^2_1|\lesssim &\bigg|\sum_{\textbf{p}\in\Pma_b^2}
   \frac{\btv_{\bfv_1}\cdot\phi_h(\textbf{p})}{h_{\bfv_1}}\Big(\int_{K_{\textbf{p}}^l}\Psi_{1}:(\sbiM-\PiHHJ\sbiM)\,dx-\int_{K_{\textbf{p}}^r}\Psi_{1}:(\sbiM-\PiHHJ\sbiM)\,dx\Big)\bigg|\\
   &+\sum_{\textbf{p}\in\Pma_b^1}\Big|\frac{\btv_{\bfv_1}\cdot\phi_h(\textbf{p})}{h_{\bfv_1}}\int_{K_{\textbf{p}}}\Psi_{1}:(\sbiM-\PiHHJ\sbiM)\,dx\Big|.
  \end{aligned}
  \end{align}
  Lemma~\ref{lem:highHHJ} and the Bramble-Hilbert lemma imply
  \begin{align}
  \begin{aligned}\label{esHHJ3}
  \big|\int_{K_{\textbf{p}}^l}\Psi_{1}:(\sbiM-\PiHHJ\sbiM)\,dx-\int_{K_{\textbf{p}}^r}\Psi_{1}:(\sbiM-\PiHHJ\sbiM)\,dx\big|
  \lesssim h^3|\sbiM|_{2,\omega_{\textbf{p}}}.
  \end{aligned}
  \end{align}
  A substitution of   \eqref{esHHJ3} and the Cauchy-Schwarz inequality into \eqref{esHHJ2} yields
  \begin{align*}
  |\bold{J}_1^2|&\lesssim h\big(\sum_{\textbf{p}\in\Pma_b^2}\|\phi_h\|^2_{0,\omega_{\textbf{p}}}\big)^{1/2}
  \big(\sum_{\textbf{p}\in\Pma_b^2}|\sbiM|_{2,\omega_{\textbf{p}}}\big)^{1/2} + h^2\big(\sum_{\textbf{p}\in\mathcal{P}_b^1}\|\phi_h\|^2_{0,\infty,K_{\textbf{p}}}\big)^{1/2}
  \big(\sum_{\textbf{p}\in\Pma_b^1}|\sbiM|_{1,K_{\textbf{p}}}\big)^{1/2}\\
  &\lesssim h\|\phi_h\|_{0,\partial_h\Omega}|\sbiM|_{2,\partial_h\Omega}+\kappa h^2\|\phi_h\|_{0,\infty,h}|\sbiM|_{1,\infty,\Omega}.
  \end{align*}
  Similar arguments as in Lemma \ref{lm:RTbdest} conclude the proof.
  \end{proof}
Similar arguments for the sums $\bold{J}_2$ and $\bold{J}_3$ prove a full one order superconvergence property for the HHJ elements as follows.
\begin{Th}\label{Thm:superHHJ}
Suppose that $(\sbiM,\ubiM)$ solves the plate bending problem \eqref{Plateequation} with $\sbiM\in H^{\frac{5}{2}}(\Omega,\Smath)$, and $(\sHHJ,\uHHJ)$ solves \eqref{MixedplateDiscrete} by the HHJ element on a uniform triangulation. It holds that
\[
\|\sHHJ-\PiHHJ\sbiM\|_{0,\Omega}\lesssim h^2
\big(|\sbiM|_{\frac{5}{2},\Omega}+\kappa|{\ln}h|^{1/2}|\sbiM|_{1,\infty,\Omega}\big).
\]
\end{Th}

\subsection{Superconvergence analysis of the Morley element}
A full one order superconvergence result for the Morley element follows from a special relation between the HHJ element and the Morley element.

Given $v\in H^2_0(\Omega,\Rmath)+ {\rm M}(\mathcal{T}_h)$, define the interpolation operator $\Pi_{\rm D}:H^2_0(\Omega,\Rmath)+ {\rm M}(\mathcal{T}_h) \rightarrow \UHHJ$ by
\begin{equation*}\label{interP1con}
  \Pi_{\rm D}v(z)=v(z) \text{ for each vertex $z$ of } \Tma.
\end{equation*}
 Hence, introduce an auxiliary method: the modified Morley element finds
 $\widetilde{u}_{\rm M}\in {\rm M}(\mathcal{T}_h)$ such that
\begin{equation}\label{PlateAxDiscrete}
  (\nabla^2_{h}\widetilde{u}_{\rm M},\nabla^2_{h}v)=(f,\Pi_{\rm D}v)\quad\text{for any }v\in {\rm M}(\mathcal{T}_h).
\end{equation}

Arnold et al. \cite{arnold1985mixed} proved the following equivalence between the HHJ element and the modified Morley element:
\begin{equation}
\label{equvilencePlate}
  \sHHJ=\nabla^2_{h}\widetilde{u}_{\rm M}\mbox{ and }\ \uHHJ=\Pi_{\rm D}\widetilde{u}_{\rm M},
\end{equation}
and moreover,
\begin{equation}
\label{DifferecePlate}
  \|\nabla_{h}^2(\uMorley-\widetilde{u}_{\rm M})\|_{0,\Omega}\lesssim h^2\|f\|_{0,\Omega}.
\end{equation}

We consider the post-processing mechanism for the superconvergence of the Morley element as in Section \ref{sec:Kdefine}. For any given $\textbf{q}\in \rm HHJ(\cT_h)$, $K_h \textbf{q}\in ({\rm CR }(\mathcal{T}_h))^{2\times 2}$.

%

Based on the the special relation \eqref{equvilencePlate}, the improved superconvergence result of the HHJ element in Theorem \ref{Thm:superHHJ} gives rise to a full one order superconvergence result of the Morley element following the procedure in \cite{Hu2016Superconvergence}. This superconvergence result improves the half order result for the Morley element in \cite{Hu2016Superconvergence}.
\begin{Th}\label{superbdPlate}
Suppose that $\ubiM\in H^{\frac{7}{2}}(\Om,\mathbb{R})\cap H^2_0(\Om,\mathbb{R})$ is the solution to \eqref{Plateequation}, $\uMorley$ is the solution to \eqref{PlateDiscrete} by the Morley element on a uniform triangulation $\cT_h$, and $f\in W^{1,\infty}(\Om,\mathbb{R})$. It holds that
$$
\parallel \nabla^2 \ubiM-K_h \nabla_h^2 \uMorley\parallel_{0,\Om}\lesssim h^2(|\ubiM|_{\frac{7}{2},\Om}+\kappa |\ln h|^{1/2}|\ubiM|_{2,\infty,\Om}+|f|_{1,\infty,\Om}).
$$
\end{Th}
\subsection{Remark for $(\alpha,\sigma)$-mesh}This subsection presents the superconvergence result on mildly structured meshes. Suppose $\mathcal{T}_h$ is a shape-regular triangulation.   For any $e\in\mathcal{E}_h$, let $h_e$ denote the length of $e$. For any boundary vertex $\textbf{p}\in\cP_b$, let $h_{\textbf{p}}:=\max_{K\subset\omega_{\textbf{p}}}h_K$ with the patch of $\omega_{\textbf{p}}$.
Recall the definitions of $O(h^{1+\alpha})$ approximate parallelograms and mildly structured meshes in \cite{bank2003asymptotically,li2017global}.
Given an interior edge $e\in\cE_h^i$,  let $K_e^1$ and $K_e^2$ be the two elements sharing $e$.  
Say $K_e^1$ and $K_e^2$ form an $O(h_e^{1+\alpha})$ approximate parallelogram if the lengths of any two opposite edges differ only by $O(h_e^{1+\alpha})$. Given a boundary vertex $\textbf{p}\in\partial\Omega$ associated with two boundary triangles $K_{\textbf{p}}^l$ and $K_{\textbf{p}}^r$ (similarly as $\textbf{p}\in\cP_b^2$ shown in Figure~\ref{fig:triangulation}), let $e_1^l$ (resp. $e_1^r$) denote the boundary edge of $K_{\textbf{p}}^l$ (resp. $K_{\textbf{p}}^r$) and $\bold{n}_1^l$ (resp. $\bold{n}_1^r$) denote its unit outnormal. By going along the boundaries of $K_{\textbf{p}}^l$ and $K_{\textbf{p}}^r$, define the other pairs of corresponding edges. Say $K_{\textbf{p}}^l$ and $K_{\textbf{p}}^r$ form an $O(h^{1+\alpha}_{\textbf{p}})$ approximate parallelogram if the lengths of any two corresponding edges differ only by $O(h_{\textbf{p}}^{1+\alpha})$ and $|\bold{n}_1^l-\bold{n}_1^r|=O(h^\alpha_{\textbf{p}})$. 

The triangulation $\cT_h$ satisfies the $(\alpha,\sigma)$-condition if the following hold:
\begin{enumerate}
\item Let $\cE_h^i=\cE_1+\cE_2$. For each $e\in\cE_1$, $K_e^1$ and $K_e^2$  form an $O(h_e^{1+\alpha})$ approximate parallelogram, while $\sum_{e\in\cE_2}|K_e^1|+|K_e^2|=O(h^{\sigma}) $.
\item Let $\cP_b=\cP_{b}^1+\cP_{b}^2$ denote the set of boundary vertices. The elements associated with each  $\bold{p}\in\cP_{b}^2$ form an $O(h_{\textbf{p}}^{1+\alpha})$ approximate parallelogram, and $|\cP_{b}^1|=\kappa$, where $\kappa$ is fixed independent of $h$.
\end{enumerate}

The result in \cite[Thm.~4.5]{li2017global} requires quasi-uniform meshes to control the number of vertices on the boundary. The analysis of this subsection only assumes that $\mathcal{T}_h$ is a regular triangulation with the following mesh-size condition
\begin{align}
\label{eq:meshcon}
\big|\ln h_K\big|\approx\big|\ln h|\text{ for all }K\in\mathcal{T}_h.
\end{align}
Under the assumption \eqref{eq:meshcon}, the discrete Sobolev inequality \eqref{eq:discreteSobolev} holds as well \cite{Bramble1986TheCO,brenner2007mathematical}.

For mildly structured meshes, the normal vectors $\bfv_i$ in \eqref{eqm} vary from different triangles. Let $\Psi_e$ denote basis function of ${\rm HHJ}(\cT_h)$ associated to $e$ with $(\Psi_e)_{\bold{n}_e\bold{n}_e}=1$.
Rewrite the terms in \eqref{eqm} as follows
  \begin{align}\label{eq:Jtotal}
  \begin{aligned}
  & (\sHHJ-\PiHHJ\sbiM,\sHHJ-\PiHHJ\sbiM)=\sum_{i=1}^3\sum_{K\in\Tma}\int_K(\sHHJ-\PiHHJ\sbiM)_{\bfv_i\bfv_i}\Psi_{i}:(\sbiM-\PiHHJ\sbiM)dx\\
   &=\bold{J}^1+  \bold{J}^2
   \end{aligned}
  \end{align}
with
\begin{align*}  \begin{aligned}
\bold{J}^1:=&\sum_{e\in\cE_h^i}(\sHHJ-\PiHHJ\sbiM)_{\bold{n}_e\bold{n}_e}\big(\int_{K_e^1}\Psi_{e}:(\sbiM-\PiHHJ\sbiM)dx+\int_{K_e^2}\Psi_{e}:(\sbiM-\PiHHJ\sbiM)dx\big),\\
  \bold{J}^2:=&\sum_{e\in\cE_h^b}(\sHHJ-\PiHHJ\sbiM)_{\bold{n}_e\bold{n}_e}\int_{K_e}\Psi_{e}:(\sbiM-\PiHHJ\sbiM)dx.   \end{aligned}
\end{align*}

  Below we only analyze the term $\bold{J}^2$ for the $(\alpha,\sigma)$-mesh. The estimate of $\bold{J}^1$ follows a similar argument. The following lemma proves a result analogous to Lemma~\ref{lem:highHHJ} for an $O(h^{1+\alpha})$ approximate parallelogram.
\begin{Lm}\label{lem:modify}
For any $\textbf{p}=(p_1, p_2)\in \cP_{b}^2$ associated with two boundary triangles $ K_{\textbf{p}}^l,K_{\textbf{p}}^r$, if $\tau_\Pma \in \sSpaMBi$ is linear on the patch $\omega_{\textbf{p}}$, then
\[
\big|\int_{K_{\textbf{p}}^l}(\tau_\Pma-\PiHHJ\tau_\Pma)\,dx-\int_{K_{\textbf{p}}^r}(\tau_\Pma-\PiHHJ\tau_\Pma)\,dx\big|\lesssim h_{\textbf{p}}^{2+\alpha} | \tau_\Pma|_{1,\omega_{\textbf{p}}} .
\] 
\end{Lm}
\begin{proof} Recall some notation  in Lemma \ref{lem:highHHJ}. 
It follows from \eqref{onlyPiHHJ} that
\begin{equation}\label{eq:reone}
\int_{K_{\bold{p}}^l} \big (\tau_\Pma(\bold{x})-\PiHHJ\tau_\Pma(\bold{x})\big )\dx = -\int_{K_{\bold{p}}^l} \PiHHJ\big ((x_1 - M_{K_{\bold{p}}^l }^1)H_1 + (x_2 - M_{K_{\bold{p}}^l }^2)H_2\big )\dx.
\end{equation}
with  $H_k|_{K_{\bold{p}}^l }=H_k|_{K_{\bold{p}}^r}$, $k=1, 2$.
Given $1\leq i\leq 3$, the basis functions of the HHJ element on elements $K_{\bold{p}}^l$ and $K_{\bold{p}}^r$ are denoted by
$$
\Psi_i^l=\frac{1}{2\big((n_i^l)^T\bold{t}^l_{i-1}\big)\big((n_i^l)^T\bold{t}^l_{i+1}\big)}\big(\bold{t}^l_{i+1}(\bold{t}^l_{i-1})^T + \bold{t}^l_{i-1}(\bold{t}^l_{i+1})^T\big)
$$
and
$$
\Psi_i^r=\frac{1}{2\big((n_i^r)^T\bold{t}^r_{i-1}\big)\big((n_i^r)^T\bold{t}^r_{i+1}\big)}\big(\bold{t}^r_{i+1}(\bold{t}^r_{i-1})^T + \bold{t}^r_{i-1}(\bold{t}^r_{i+1})^T\big),
$$
respectively. Let $\bold{m}_{i}^l=(m_{i,1}^l,m_{i,2}^l)$.  Since $(\Psi_i^l)_{\bold{n}^l_j\bold{n}^l_j}=\delta_{ij}$,   \eqref{eq:reone} shows
\begin{align}\label{onlyPiHHJtwo}
\int_{K_{\bold{p}}^l} \big (\tau_\Pma(\bold{x})-\PiHHJ\tau_\Pma(\bold{x})\big )\dx=- |K_{\bold{p}}^l|  \sum_{i=1}^3\big((m_{i,1}^l-M_{K_{\bold{p}}^l}^1)H_1+(m_{i,2}^l-M_{K_{\bold{p}}^l}^2)H_2)_{\bold{n}_i^l\bold{n}_i^l}\Psi_i^l.
\end{align} Since $K_{\bold{p}}^l$ and $K_{\bold{p}}^r$ form an $O(h^{1+\alpha}_{\textbf{p}})$ approximate parallelogram, this leads to 
\begin{align*}
|K_{\bold{p}}^l-K_{\bold{p}}^r|\lesssim h^{\alpha}_{\textbf{p}}|K_{\bold{p}}^l|,\quad |(\bold{m}_i^l-\bold{M}_{K_{\bold{p}}^l})-(\bold{m}_i^r-\bold{M}_{K_{\bold{p}}^r })|\lesssim h^{\alpha}_{\textbf{p}}|\bold{m}_i^l-\bold{M}_{K_{\bold{p}}^l}|,\\
|\Psi_i^l-\Psi_i^r|\lesssim h^{\alpha}_{\textbf{p}}|\Psi_i^l|,\quad |(\bold{n}_i^l)^TH_k\bold{n}_i^l-(\bold{n}_i^r)^TH_k\bold{n}_i^r|\lesssim h^{\alpha}_{\textbf{p}}|H_k|.
\end{align*} 
The combination of the above estimates  with \eqref{onlyPiHHJtwo} leads to
\begin{align*}
\big|\int_{K_{\textbf{p}}^l}(\tau_\Pma-\PiHHJ\tau_\Pma)\,dx-\int_{K_{\textbf{p}}^r}(\tau_\Pma-\PiHHJ\tau_\Pma)\,dx\big| \lesssim h^{\alpha}_{\textbf{p}}|K_{\bold{p}}^l||\bold{m}_i^l-\bold{M}_{K_{\bold{p}}^l}||\Psi_i^l|(|H_1|+|H_2|) 
 \lesssim h^{2+\alpha}_{\textbf{p}}|\tau_\Pma|_{1,\omega_{\textbf{p}}}.
\end{align*}
This concludes the proof.
\end{proof}
 \begin{Lm}\label{lem:Jonetwomodify}Let  $\rho=\min(1,\alpha)$ and assume $\sbiM\in H^{\frac{5}{2}}(\Omega)$. Then,
  \begin{align*}
 | \bold{J}^2|\lesssim h^{1+\rho}(|\sbiM|_{\frac{5}{2},\Omega}+\kappa|\ln h|^{1/2}|\sbiM|_{1,\infty,\Omega})\|\sHHJ-\PiHHJ\sbiM\|_{0,\Omega}.
  \end{align*}
  \end{Lm}
  \begin{proof}The proof is analogous to that of  Lemma~\ref{lem:Jonetwo}.   Let $\textbf{p}_e^1$ and $\textbf{p}_e^2$ denote the two vertices of $e$ such that $\bold{t}_e$ points from $\textbf{p}_e^1$ to $\textbf{p}_e^2$. The same arguments as in \eqref{eq:Jonetwofirst}-\eqref{eq:Jonetwosec} lead to \begin{align*} 
   \bold{J}^2=&\sum_{e\in\cE_h^b}(H^T\epsilon(\phi_h)H)_{\bold{n}_e\bold{n}_e}\int_{K_{e}}\Psi_{e}:(\sbiM-\PiHHJ\sbiM)\,dx 
  \end{align*}
and
  \begin{align*} 
 (H^T\epsilon(\phi_h)H)_{\bold{n}_e\bold{n}_e}=\frac{1}{|e|}\int_{e }\nabla(\bold{t}_e\cdot\phi_h)\cdot \bold{t}_e\,dx=\frac{\bold{t}_e\cdot\phi_h(\textbf{p}_{e}^2)-\bold{t}_e\cdot\phi_h(\textbf{p}_{e}^1)}{|e|}.
  \end{align*}
  with $\phi_h\in (S_h)^2/{\rm RM}$ in \eqref{eq:prokey}.
Given any  $\textbf{p}\in\cP_{b}$, recall its  two  associated boundary triangles $K_{\textbf{p}}^l$ and $K_{\textbf{p}}^r$. It holds that
\begin{align} \label{eq:Jtwoes}
  | \bold{J}^2|\leq&\sum_{\textbf{p}\in\Pma_b} \big|\frac{ t_{e_1^l} \cdot\phi_h(\textbf{p})}{|e_1^l|}\int_{K_{\textbf{p}}^l}\Psi_{e_1^l}:(\sbiM -\hspace{-0.25mm}\PiHHJ\sbiM)\,dx - \frac{ t_{e_1^r} \cdot\phi_h(\textbf{p})}{|e_1^r|}\int_{K_{\textbf{p}}^r}\Psi_{e_1^r}:(\sbiM\hspace{-0.3mm}-\hspace{-0.3mm}\PiHHJ\sbiM)\,dx \big|
  \end{align}
Since   the elements associated with each  $\bold{p}\in\cP_{b}^2$ form an $O(h^{1+\alpha}_{\textbf{p}})$ approximate parallelogram, this shows $| t_{e_1^l}- t_{e_1^r}|\lesssim h_{\textbf{p}}^\alpha$, $\big||e_1^l|-|e_1^r|\big|\lesssim h_{\textbf{p}}^{ \alpha}|e_1^l|$ and $|\Psi_{e_1^l}-\Psi_{e_1^r}|\lesssim h_{\textbf{p}}^{\alpha}$. The combination with \eqref{eq:Jtwoes} and the estimate of the interpolation $\Pi_{\rm HHJ}$ analogy to \eqref{interpolationHHJ} on each element
   \begin{align*} 
   \begin{aligned}
   &|\bold{J}^2| \lesssim \sum_{\textbf{p}\in\Pma_b^2} |\phi_h(\textbf{p})| \Big( h_{\textbf{p}}^{1+\alpha} |\sbiM|_{1,K_{\textbf{p}}^r}+|e_1^l|^{-1} \Big|
  \int_{K_{\textbf{p}}^l}\Psi_{e_1^l}:(\sbiM\hspace{-0.25mm}-\hspace{-0.25mm}\PiHHJ\sbiM)\,dx\hspace{-0.25mm}-\hspace{-0.25mm}\int_{K_{\textbf{p}}^r}\Psi_{e_1^l}:(\sbiM\hspace{-0.3mm}-\hspace{-0.3mm}\PiHHJ\sbiM)\,dx \Big|\Big)\\
   &+  \sum_{\textbf{p}\in\Pma_b^1 }|e_1^l|^{-1}|\phi_h(\textbf{p})| \Big(\big|
  \int_{K_{\textbf{p}}^l}\Psi_{e_1^l}:(\sbiM-\PiHHJ\sbiM)\,dx\big|+\big|\int_{K_{\textbf{p}}^r}\Psi_{e_1^r}(\sbiM-\PiHHJ\sbiM)\,dx \big|\Big).
  \end{aligned}
  \end{align*}
Recall $\rho=\min(1,\alpha)$ and let $\Pi_1$ denote the $L^2$ projection onto  linear polynomials on $\omega_{\textbf{p}}$. A triangle inequality, the approximation of $\Pi_1$  and Lemma~\ref{lem:modify} modify the result in \eqref{esHHJ3} as follows 
 \begin{align*}
  \begin{aligned}
 &\big|\int_{K_{\textbf{p}}^l}\Psi_{e_1^l}:(\sbiM-\PiHHJ\sbiM)\,dx-\int_{K_{\textbf{p}}^r}\Psi_{e_1^l}:(\sbiM-\PiHHJ\sbiM)\,dx\big|\\
  \lesssim&h_{\textbf{p}}^3|\sbiM|_{2,\omega_{\textbf{p}}}+ \big|\int_{K_{\textbf{p}}^l}\Psi_{e_1^l}:(\Pi_1\sbiM-\PiHHJ(\Pi_1\sbiM))\,dx-\int_{K_{\textbf{p}}^r}\Psi_{e_1^l}:(\Pi_1\sbiM-\PiHHJ(\Pi_1\sbiM))\,dx\big|\\
  \lesssim &h_{\textbf{p}}^{2+\rho}\|\sbiM\|_{2,\omega_{\textbf{p}}}.
  \end{aligned}
  \end{align*}
  The remaining arguments in Lemma~\ref{lem:Jonetwo}  conclude the proof.
\end{proof}
\begin{theorem}
 If $\cT_h$ satisfies the $(\alpha,\sigma)$-condition and the meshsize condition \eqref{eq:meshcon}, then
 \begin{align}\label{thm:resu}
\|\sHHJ-\PiHHJ\sbiM\|_{0,\Omega}\lesssim h^{1+\rho}
\big(\|\sbiM\|_{\frac{5}{2},\Omega}+\kappa|{\ln}h|^{1/2}|\sbiM|_{1,\infty,\Omega}\big)
\end{align}
with $\rho=\min\big(1,\alpha,\frac{\sigma}{2}\big)$.
\end{theorem}
 \begin{proof}
 The estimate of the term for $e\in\cE_1 $  in $\bold{J}^1$  on the right-hand side of \eqref{eq:Jtotal} combines   Theorem~ 5.3 of \cite{Hu2016Superconvergence} for uniform meshes  and  a similar argument as in Lemma~\ref{lem:modify} for $O(h^{1+\alpha})$ approximate parallelograms. The details are omitted.  It holds  
   \begin{align*}
   &\sum_{e\in\cE_1}(\sHHJ-\PiHHJ\sbiM)_{\bold{n}_e\bold{n}_e}\big(\int_{K_e^1}\Psi_{e}:(\sbiM-\PiHHJ\sbiM)dx+\int_{K_e^2}\Psi_{e}:(\sbiM-\PiHHJ\sbiM)dx\big)\\
   \lesssim &h^{1+\min (1,\alpha)}\|\sHHJ-\PiHHJ\sbiM\|_{0,\Omega}\|\sbiM\|_{2,\Omega}.
  \end{align*}
For $e\in\cE_2$,  the estimate of the  interpolation $\Pi_{\rm HHJ}$ shows
 \begin{align*}
 &\sum_{e\in\cE_2}(\sHHJ-\PiHHJ\sbiM)_{\bold{n}_e\bold{n}_e}\big(\int_{K_e^1}\Psi_{e}:(\sbiM-\PiHHJ\sbiM)dx+\int_{K_e^2}\Psi_{e}:(\sbiM-\PiHHJ\sbiM)dx\big)\\
\lesssim& \sum_{e\in\cE_2}h_e(|K_e^1|^{1/2}+|K_e^2|^{1/2})\|\sHHJ-\PiHHJ\sbiM \|
 _{K_e^1}|( | \sbiM|_{1,\infty,K_e^1}+ | \sbiM|_{1,\infty,K_e^2})\\
\lesssim&h^{1+\frac{\sigma}{2}}\|\sHHJ-\PiHHJ\sbiM\|_{0,\Omega}  | \sbiM|_{1,\infty,\Omega}.
 \end{align*} Hence,
   \begin{align*}
   |\bold{J}^1|\lesssim h^{1+\min (1,\alpha,\frac{\sigma}{2})}\|\sHHJ-\PiHHJ\sbiM\|_{0,\Omega}(\|\sbiM\|_{2,\Omega}+| \sbiM|_{1,\infty,\Omega}).
  \end{align*}
  The combination with \eqref{eq:Jtotal}  and Lemma~\ref{lem:Jonetwomodify} concludes the proof.
 \end{proof}
 \begin{remark}Let $\Omega$ be decomposed into $N$ subdomains, where $N$ is independent of $h$.  $\mathcal{T}_h$ satisfies the piecewise $(\alpha,\sigma)$-condition if the restriciton of $\mathcal{T}_h$ to each subdomain satisfies the $(\alpha,\sigma)$-condition. As remarked in \cite[Remark 3.8]{li2017global},  \eqref{thm:resu} still holds on piecewise $(\alpha,\sigma)$-grids.
 \end{remark}
\begin{remark} Suppose $\cT_h$ satisfies the $(\alpha,\sigma)$-condition and \eqref{eq:meshcon}. Following similar  arguments, 
a superconvergence result analogous to Theorem~\ref{Lm:bdRT} shall hold  for the Raviart-Thomas element, i.e.,
$$
\parallel \sigRTS-\PiRT\sigS\parallel_{0,\Om}\lesssim h^{1+\rho} \big (| \sigS|_{\frac{5}{2},\Om}+\kappa |\ln h|^{1/2}|\sigS|_{1,\infty,\Om}\big ).
$$
The details are omitted.
\end{remark}
\section{Numerical examples}
In this section, we present  numerical tests for the  Morley element  to confirm the theoretical superconvergence analysis in Section~\ref{sec:HHJ}. The results for RT and CR  elements  can be found in \cite{Hu2016Superconvergence, li2017global}. Consider the following plate bending problem
\begin{equation*}
  \Delta^2 \ubiM=f\quad\text{in }\Omega
\end{equation*}
with $\ubiM\in H^2_0(\Omega)$. The exact solution is
\begin{equation*}
  \ubiM(x_1,x_2)=(x_2-\sqrt{3}x_1)^2(x_2-\sqrt{3}x_1+2\sqrt{3})^2x_2^2( \sqrt{3} -x_2)^2.
\end{equation*}
We compare the error $\|\nabla^2 \ubiM-\nabla_h^2\uMorley\|_{0,\Omega}$, the interpolation error  $\| \nabla_h^2\uMorley-\PiHHJ\nabla^2 \ubiM\|_{0,\Omega}$ and the post-processing error $\|\nabla^2 \ubiM-K_h\nabla_h^2\uMorley\|_{0,\Omega}$ by the Morley element. 

Figure~\ref{fig:para}  shows three kinds of initial meshes. The meshes are generated by uniform refinements. The corresponding results are listed in Table~\ref{platesuperone}-\ref{platesuperthree}. The initial mesh in Figure~\ref{fig:sub2a} is a uniform mesh.  The results in Table~\ref{platesuperone} coincides with the theoretical results. The initial mesh in Figure~\ref{fig:sub2b} is a  $(\infty,1)$-mesh. However, it is a piecewise uniform mesh and in this case $\Omega$ is decomposed into two subdomains. Table~\ref{platesupertwo} shows that the interpolation error is optimal while the postprocessing error is suboptimal. This happens because  the values  of the postprocessing $K_h$ on the edges of the boundary of subdomains are not chosen appropriately. The initial mesh in Figure~\ref{fig:sub2c} is a delaunay mesh. Since the meshes are generated by uniform refinements, the initial mesh partitions $\Omega$ into several subdomains and its refinement are piecewise uniform meshes. The result in Table~\ref{platesuperthree} is similar to that in Table~\ref{platesupertwo}.
\begin{figure}[htbp]                                                 
\subfigure[uniform mesh]{                    
\begin{minipage}{4cm}\centering                                                          
\includegraphics[width=4cm,height=2.7cm]{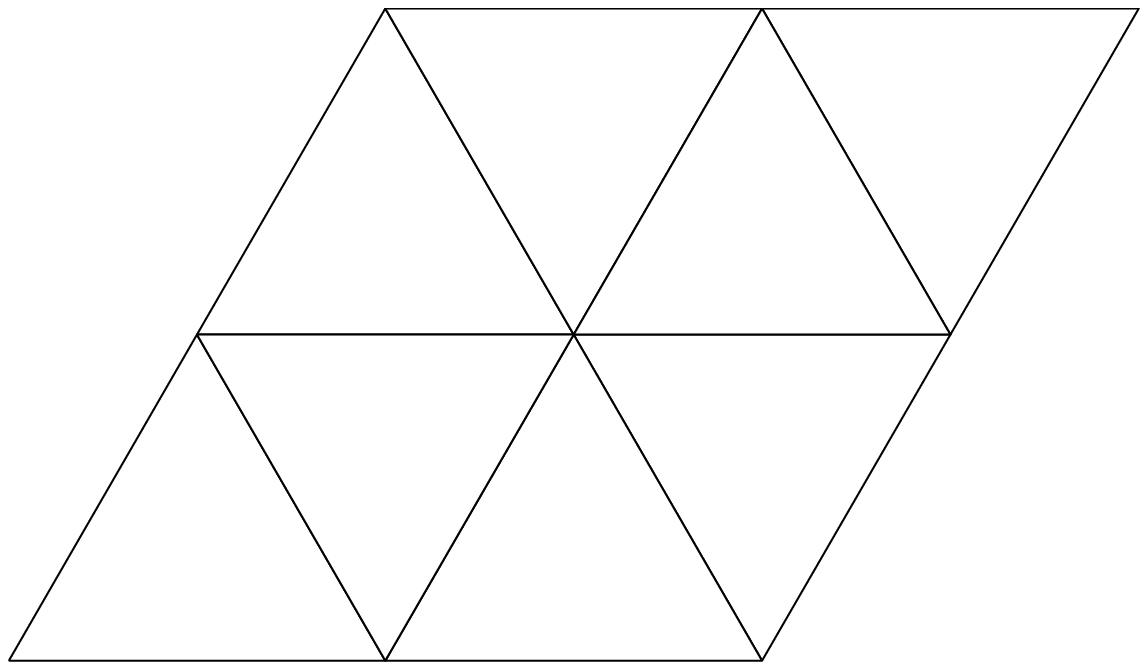}         \label{fig:sub2a}          
\end{minipage}}
\subfigure[piecewise uniform mesh]{                    
\begin{minipage}{4cm}\centering                                                          
\includegraphics[width=4cm,height=2.7cm]{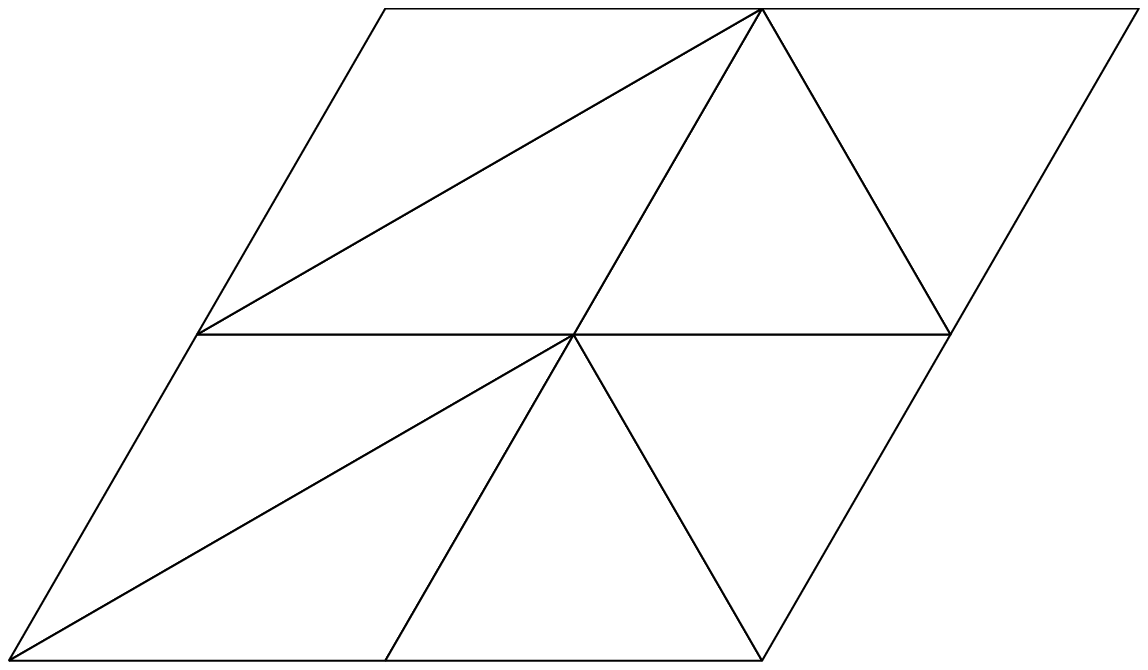}         \label{fig:sub2b}          
\end{minipage}}\subfigure[delaunay  mesh]{                    
\begin{minipage}{4cm}\centering                                                          
\includegraphics[width=4cm,height=2.7cm]{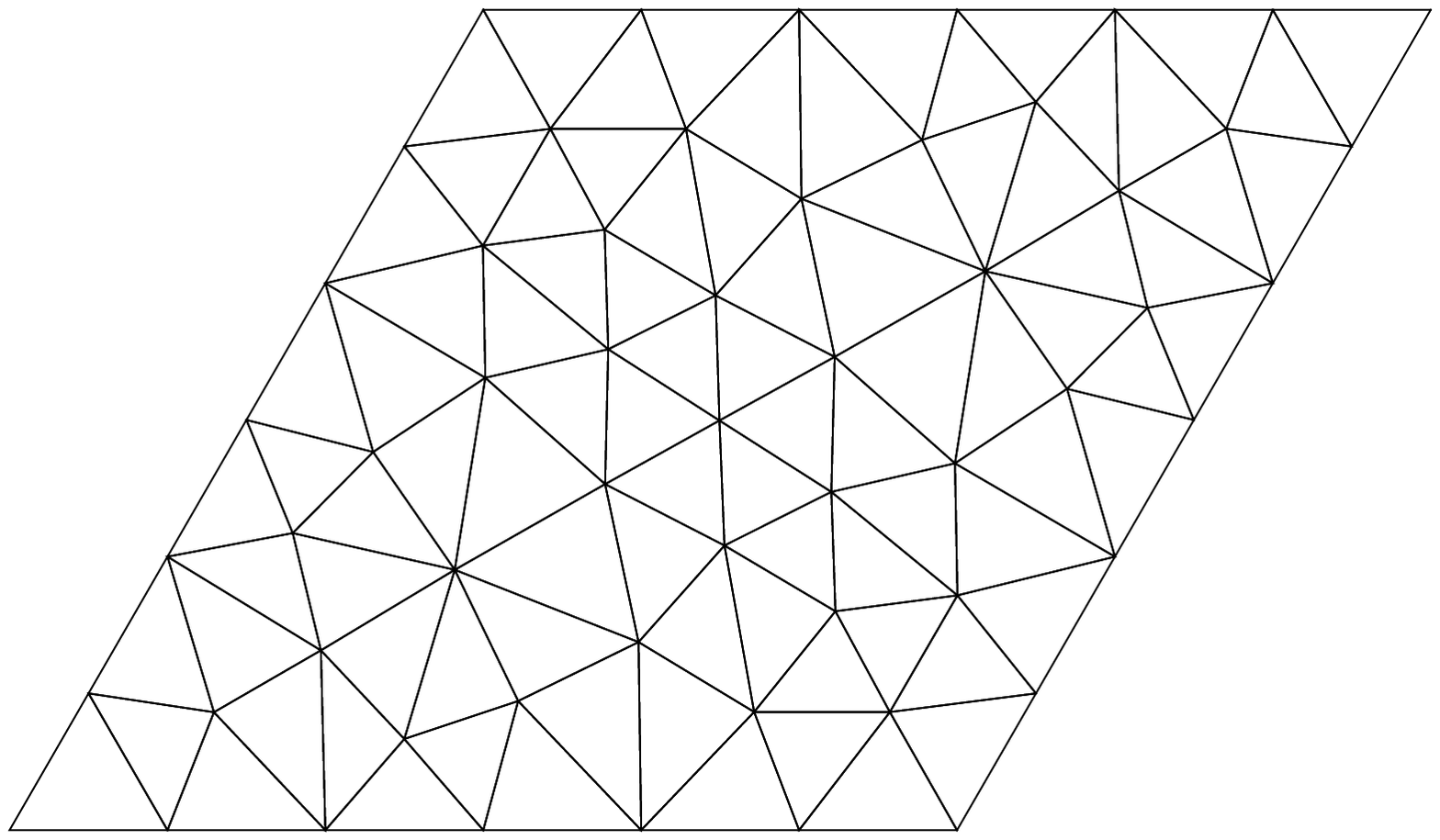}        \label{fig:sub2c}      
\end{minipage}   }
\caption{Initial meshes} 
\label{fig:para}                                                        
\end{figure}

\begin{table}[!ht]
\footnotesize
  \caption{Convergence results on mesh (a)}\label{platesuperone}
  \centering
\begin{tabular}{|c|c|c|c|c|c|c|}
\hline
Mesh &$\|\nabla^2 \ubiM-\nabla_h^2\uMorley\|_{0,\Omega}$   &   $Rate$&    $\| \nabla_h^2\uMorley-\PiHHJ\nabla^2 \ubiM\|_{0,\Omega}$          &    $Rate$&    $\|\nabla^2 \ubiM-K_h\nabla_h^2\uMorley\|_{0,\Omega}$          &    $Rate$  \\
\hline
1	&	61.9148	&		&	68.6151	&		&	67.8685	&	\\
2	&	40.1923	&	0.6234	&	21.4424	&	1.6781	&	27.3979	&	1.3087\\
3	&	23.6110	&	0.7675	&	5.8356	&	1.8775	&	8.3891	&	1.7075\\
4	&	12.3520 &	0.9347	&	1.4922	&	1.9674	&	2.1181	&	1.9857\\
5	&	6.2484	&	0.9832	&	0.3752	&	1.9917	&	0.5147	&	2.0410\\
6	&	3.1334	&	0.9958	&	0.0939	&	1.9985	&	0.1252	&	2.0395\\
7	&	1.5678	&	0.9990	&	0.0235	&	1.9985	&	0.0307	&	2.0279\\
\hline
\end{tabular}
\end{table}
\begin{table}[!ht]
\footnotesize
  \caption{Convergence results on mesh (b)}\label{platesupertwo}
  \centering
\begin{tabular}{|c|c|c|c|c|c|c|}
\hline
Mesh &$\|\nabla^2 \ubiM-\nabla_h^2\uMorley\|_{0,\Omega}$   &   $Rate$&    $\| \nabla_h^2\uMorley-\PiHHJ\nabla^2 \ubiM\|_{0,\Omega}$          &    $Rate$&    $\|\nabla^2 \ubiM-K_h\nabla_h^2\uMorley\|_{0,\Omega}$            &    $Rate$  \\
\hline
1&		70.3165	&		&	122.4437	&		&	75.6681	&	\\
2&		62.2872	&	0.1749	&	44.2756	&	1.4675	&	41.5546	&	0.8647\\
3	&	38.7694	&	0.6840	&	13.9626	&	1.6649	&	15.186	&	1.4523\\
4	&	20.7579	&	0.9013	&	4.0793	&	1.7752	&	4.7578	&	1.6744\\
5&		10.5964	&	0.9701	&	1.1357	&	1.8447	&	1.4481	&	1.7161\\
6	&	5.3302	&	0.9913	&	0.3067	&	1.8887	&	0.4432	&	1.7081\\
7	&	2.6696	&	0.9976	&	0.0815	&	1.9120	&	0.1395	&	1.6677\\
\hline
\end{tabular}
\end{table}
\begin{table}[!ht]
\footnotesize
  \caption{Convergence results on mesh (c)}\label{platesuperthree}
  \centering
\begin{tabular}{|c|c|c|c|c|c|c|}
\hline
Mesh &$\|\nabla^2 \ubiM-\nabla_h^2\uMorley\|_{0,\Omega}$   &   $Rate$&     $\| \nabla_h^2\uMorley-\PiHHJ\nabla^2 \ubiM\|_{0,\Omega}$          &    $Rate$&    $\|\nabla^2 \ubiM-K_h\nabla_h^2\uMorley\|_{0,\Omega}$         &    $Rate$  \\
\hline
1&29.3128	& &	11.1055	& &	15.3555	&\\
2&15.8475	&0.8873	&3.1380	&1.8234	&5.0514	&1.6040\\
3&8.1044	&0.9675&		0.8576	&	1.8715	&1.6318	&1.6302\\
4&4.0780&	0.9908&		0.2296&		1.9012	&0.5309	&1.6200\\
5&2.0426	&0.9975&		0.0607&		1.9194	&0.1770	&1.5847\\
6&1.0218	&0.9993&		0.0160&		1.9236&	0.0603&	1.5535\\
\hline
\end{tabular}
\end{table}


\end{document}